\documentclass[11pt, a4paper]{amsart}
\usepackage[dvips]{graphicx}
\usepackage{amssymb, amstext, amscd, amsmath, color}

\usepackage{tikz,mathdots,enumerate}
\usepackage{pgfplots}%
\usepackage{url}
\usetikzlibrary{arrows}

\definecolor{cof}{RGB}{219,144,71}
\definecolor{pur}{RGB}{186,146,162}
\definecolor{greeo}{RGB}{91,173,69}
\definecolor{greet}{RGB}{52,111,72}

\usepackage{hhline}
\usepackage[%
]{hyperref}
\usepackage{xcolor}
\hypersetup{
    colorlinks,
    linkcolor={red!50!black},
    citecolor={blue!50!black},
    urlcolor={blue!80!black}
}

\newtheorem{theorem}{Theorem}[section]
\newtheorem{corollary}[theorem]{Corollary}
\newtheorem{proposition}[theorem]{Proposition}

\newtheorem{lemma}[theorem]{Lemma}
\theoremstyle{definition}
\newtheorem{remark}[theorem]{Remark}
\newtheorem{definition}[theorem]{Definition}
\newtheorem{example}[theorem]{Example}
\newtheorem{conjecture}[theorem]{Conjecture}

\newcommand{\unorm}[1]{\|#1\|} 
\newcommand{\transpose}[1]{#1^{\intercal}}

\DeclareMathOperator{\trace}{trace}
\DeclareMathOperator{\id}{id}
\DeclareMathOperator{\Fix}{Fix}

\newcommand{\bC}{{\mathbb{C}}}

\newcommand{\bN}{{\mathbb{N}}}

\newcommand{\bR}{{\mathbb{R}}}
\newcommand{\bF}{{\mathbb{F}}}
\newcommand{\bT}{{\mathbb{T}}}

 \newcommand{\A}{{\mathcal{A}}}

 \newcommand{\F}{{\mathcal{F}}}
 
\renewcommand{\H}{{\mathcal{H}}}

 \newcommand{\M}{{\mathcal{M}}}

 \newcommand{\R}{{\mathcal{R}}}
\providecommand{\T}{}\renewcommand{\T}{{\mathcal{T}}}
  \def\U{{\mathcal{U}}}

\newcommand{\range}{\operatorname{ran}}

\newcommand{\rank}{\operatorname{rank}}
\newcommand{\Isom}{\operatorname{Isom}}

\newcommand{\tr}{\operatorname{trace}}
\newcommand{\Skew}{\operatorname{Skew}}
\newcommand{\diag}{\operatorname{diag}}

\newcommand{\qtext}[1]{\quad\text{#1}\quad}
\newcommand{\qand}{\qtext{and}}

\title[Graph rigidity for unitarily invariant matrix norms]{Graph rigidity for unitarily invariant matrix norms}

\author[D. Kitson]{Derek Kitson}
\email{d.kitson@lancaster.ac.uk}
\address{Dept.\ Math.\ Stats.\\ Lancaster University\\
Lancaster LA1 4YF \\U.K. }
\thanks{The first named author is supported by EPSRC grant  EP/P01108X/1.}
\author[R.\,H. Levene]{Rupert H. Levene}
\email{rupert.levene@ucd.ie}
\address{School of Mathematics and Statistics\\University College Dublin\\Belfield\\Dublin 4\\Ireland}%

\let\qedhere\relax


\begin{document}
\maketitle
\begin{abstract}
A rigidity theory is developed for bar-joint frameworks in linear matrix spaces endowed with a unitarily invariant norm. Analogues of Maxwell's counting criteria are obtained and minimally rigid matrix frameworks are shown to belong to the matroidal class of $(k,l)$-sparse graphs for suitable $k$ and $l$. 
A characterisation of infinitesimal rigidity is obtained for product norms and it is shown that $K_6-e$ (respectively, $K_7$) is the smallest minimally rigid graph for the class of $2\times 2$ symmetric (respectively, hermitian) matrices with the trace norm. 
\end{abstract}
\tableofcontents

\section{Introduction}
A {\em bar-joint framework} is a pair $(G,p)$ consisting of a simple undirected graph $G=(V,E)$ and a mapping of its vertices $p:V\to X$ into a linear space $X$, with $p(v)$ and $p(w)$ distinct for each edge $vw\in E$. Given such a framework, and a norm on $X$, one may ask whether it is possible to perturb the elements of $p(V)$ without altering distances between adjacent vertices, and without applying an isometry of $X$ to $p(V)$. 
This generalises to the setting of normed linear spaces a central problem in structural rigidity for Euclidean bar-joint frameworks; a topic with roots in works of Cauchy \cite{cau} and Maxwell \cite{maxwell} and a broad spectrum of applications (see for example \cite{tho-dux,gue-fow-pow}). 
Recently,  aspects of graph rigidity  have been investigated for polyhedral and $\ell^p$ norms and in general normed spaces \cite{kitson,kit-pow,kit-sch}. 
In this article, we develop  matricial graph rigidity for bar-joint frameworks in linear matrix spaces endowed with a unitarily invariant norm. 
Interesting physical interpretations arise in quantum information theory where Schatten $p$-norms (in particular, the trace norm) feature in the representation of quantum states.
 For an introduction to graph rigidity in Euclidean space we refer the reader to \cite{asi-rot,gra-ser-ser,whi84}.

In Section \ref{RigidMotions} we identify rigid motions for a class of admissible matrix spaces.  This class includes the spaces of all $n\times n$ real and complex matrices, the $n\times n$ symmetric matrices and the  $n\times n$ hermitian matrices.  We then characterise the infinitesimal rigid motions for these spaces (Theorem \ref{thm:skew}) and, in Section \ref{GraphRigidity}, present a rank formula which characterises infinitesimal rigidity for certain matrix frameworks which we call \emph{full} (including those with full affine span). We then provide analogues of the Maxwell counting criteria for Euclidean bar-joint frameworks (Theorem \ref{thm:maxwell}) and show that the graphs of minimally rigid matrix frameworks belong to the matroidal class of $(k,l)$-sparse graphs for suitable values of $k$ and $l$ (Theorem \ref{thm:tight}). Such graphs satisfy a counting rule which is checkable by existing polynomial-time pebble game algorithms.  Interactions between the algebraic structure of these matrix spaces and the accompanying rigidity theory emerge both in the determination of rigid motions and in the identification of infinitesimal flexes for matrix frameworks. 

In Section~\ref{ProductNorms} we obtain a geometric characterisation of infinitesimal rigidity for product norms (Theorem \ref{ProjThm}). 
(This result may be of independent interest). We apply this characterisation in Section~\ref{sect:applications}, where we exploit the cylindrical nature of the trace norm on the space of $2\times 2$ symmetric matrices, to show that the graph of a minimally rigid matrix framework is expressible as an edge-disjoint union of a spanning tree and a spanning Laman graph (Theorem \ref{thm:cyl}). We then exhibit a minimally rigid matrix framework for the smallest such graph, the complete graph $K_6$ with an edge removed, and show that a complete graph $K_m$ admits a placement as a rigid matrix framework if and only if $m\geq 6$. Analogous results are obtained for the space of $2\times 2$ hermitian matrices.

\subsection{Preliminaries}

We now recall a few standard definitions and fix some notation. Throughout, we let $n\in\bN$ with $n\ge2$. Let $\bF$ be either $\bR$ or $\bC$ and let $\M_n(\bF)$ denote the associative algebra of $n\times n$ matrices over $\mathbb F$. As usual, we write~$a^*$ for the conjugate transpose, or adjoint, of a
matrix~$a\in \M_n(\bF)$ (which is simply the transpose in the real case).
Let $\U_n(\bF)$, $\H_n(\bF)$ and $\Skew_n(\bF)$ denote respectively the sets of unitary, hermitian and skew-hermitian matrices in $\M_n(\bF)$  (which in the real case are simply the orthogonal, symmetric and skew-symmetric matrices). We also write $\Skew_n^0(\bF)$ for the set of skew-hermitian matrices with a zero in the $(1,1)$ entry; note that $\Skew_n^0(\bR)=\Skew_n(\bR)$ and $\Skew_n^0(\bC)\subsetneq\Skew_n(\bC)$.
Recall that the commutant $S'$ of a set $S\subseteq \M_n(\bF)$ is the unital algebra
\[ S'=\{y\in \M_n(\bF)\colon \forall\,x\in S,\ xy=yx\}.\] 
For
$x,y\in \M_n(\bF)$, the commutator of $x$ and
$y$ is $[x,y]=xy-yx$. %
If $x=(x_1,\ldots, x_n)\in\bF^n$ then $\diag(x)$ denotes the diagonal matrix in~$\M_n(\bF)$ whose $i$th diagonal entry is~$x_i$.

A norm $\|\cdot\|$ on $\M_n(\bF)$ is \emph{unitarily invariant} if
\[\|a\| = \|uaw\|\quad  \forall \,a\in \M_n(\bF),\,\, \forall \,u,w\in \U_n(\bF).\]
A norm $\|\cdot\|_s$ on $\bR^n$ is \emph{symmetric} if $\|(x_1,\ldots,x_n)\|_s = \|(|x_{\pi(1)}|,\ldots,|x_{\pi(n)}|)\|_s$ for all $(x_1,\ldots, x_n)\in\bR^n$ and all permutations $\pi\in S(n)$.
Von Neumann~\cite{neu} characterised unitarily invariant matrix norms on $\M_n(\bF)$ as those obtained by applying a symmetric norm $\|\cdot\|_s$ to the vector \[\sigma(a)=(\sigma_1(a),\ldots,\sigma_n(a)),\] where $\sigma_i(a)$ is the $i$th largest singular value of the matrix~$a\in \M_n(\bF)$. %
The correspondence is given by
\[\|a\| := \|\sigma(a)\|_s, \quad \|x\|_s := \|\diag(x)\|.\]
Standard examples of unitarily invariant norms are provided by the Schatten $p$-norms
\[\|a\|_{c_p}:= \|\sigma(a)\|_{\ell_p}, \quad \forall \, 1\leq p\leq\infty,\]
and the Ky-Fan $k$-norms
\[\|a\|_{k}:= \sum_{i=1}^k \sigma_i(a),\quad \forall \, 1\leq k\leq n.\]
The Schatten $1$-norm, $2$-norm and $\infty$-norm are known as the trace norm, the Frobenius norm and the spectral norm, respectively. The Frobenius norm is Euclidean in the sense that it is derived from an inner product. The spectral norm is an operator norm with matrices viewed as linear operators on~$\bF^n$ with the usual Euclidean norm.

\section{Rigid motions for admissible matrix spaces}
\label{RigidMotions}
The aim of this section is to describe the linear space of infinitesimal rigid motions for a rich class of normed matrix spaces. 
Explicit characterisations are obtained for suitable norms in the cases of $\M_n(\bF)$ and $\H_n(\bF)$.

\subsection{Admissible matrix spaces}

Let $\Gamma$ be a finite set of real-linear maps  $\M_n(\bF)\to\M_n(\bF)$ which
contains the identity map~$\id$, and has the property that
$\gamma(I)=I$ for all $\gamma\in \Gamma$; we call such a set~$\Gamma$
a \emph{test set} on~$\M_n(\bF)$.
Let $X$ be a real-linear subspace of~$\M_n(\bF)$. If~$\gamma\in
\Gamma$, then the \emph{$\gamma$-commutant} of~$X$ is the real-linear
subspace \[X^\gamma=\{y\in \M_n(\bF)\colon \forall\,x\in X,\
xy=y\gamma(x)\},\] and we
define \[X^\Gamma=\bigcup_{\gamma\in\Gamma}X^\gamma.\] Note that
$X^\Gamma$ decreases as~$X$ increases, and \[X^\Gamma\supseteq
X^{\id}=X'\supseteq \bF I= \{\lambda I\colon \lambda\in \bF\}.\]
\begin{definition}
  If $I\in X$ and $X^\Gamma=\bF I$ is as small as possible, then we
  say that~$X$ is \emph{$\Gamma$-large} in $\M_n(\bF)$.
\end{definition}
\begin{remark}\label{remark:large}
  Let $\Fix(X;\Gamma)$ be the set of matrices in~$X$ fixed
  by a test set~$\Gamma$:
  \[\Fix(X;\Gamma)=\{x\in X\colon \forall\,\gamma\in
  \Gamma,\ \gamma(x)=x\}.\] Plainly, $X^\Gamma\subseteq
  \Fix(X;\Gamma)'$. In particular, if $e_{ij}$ denotes the $(i,j)$ matrix unit in
  $\M_n(\bF)$ and \[S:=\{e_{ij}+e_{ji}\colon 1\le i\le j\le n\}
  \subseteq \Fix(X;\Gamma),\] then $X^\Gamma\subseteq S'=\bF I$, so
  $X$ is $\Gamma$-large in $\M_n(\bF)$.
\end{remark}

\begin{example}
  \label{ex:large}
  Consider \[\Gamma_\bR=\{\text{identity},\text{transpose}\}\quad\text{and}\quad
  \Gamma_\bC=\Gamma_\bR\cup\{\text{adjoint},\text{conjugation}\}.\]
  Plainly, $\Gamma_\bF$ is then a test set on~$\M_n(\bF)$. It is easy
  to check using Remark~\ref{remark:large} that the real-linear spaces
  $\H_n(\bR)$, $\H_n(\bC)$ and $\M_n(\bR)$ are $\Gamma_\bR$-large, and
  $\M_n(\bC)$ is $\Gamma_\bC$-large, in the corresponding $\M_n(\bF)$.
\end{example}

\begin{definition}
  \begin{enumerate}
  \item Let~$\Gamma$ be a test set on~$\M_n(\bF)$ and let
    $\unorm\cdot$ be a unitarily invariant norm on $\M_n(\bF)$.  A
    real-linear subspace $(X,\unorm\cdot)$ of $\M_n(\bF)$ has the
    \emph{$\Gamma$-isometry property} if every real-linear isometry
    $A\colon X\to X$ is of the form
    \[A(x)=u\,\gamma(x)\,w,\quad x\in X\] for some $u,w\in \U_n(\bF)$,
    and some $\gamma\in \Gamma$.

  \item   Given a real-linear space $X\subseteq \M_n(\bF)$ and a unitarily invariant norm
  $\|\cdot\|$ on~$\M_n(\bF)$, we call $(X,\|\cdot\|)$ an \emph{admissible matrix
    space} (in~$\M_n(\bF)$) if
  \begin{enumerate}
  \item there exists a test set $\Gamma$ such that $X$ is $\Gamma$-large in
    $\M_n(\bF)$ and $(X,\|\cdot\|)$ has the $\Gamma$-isometry
    property; and
  \item there exist scalars~$\lambda_i\in \bF$ for $1\leq i\leq n$ so that $e_{ii}\in X$ and $e_{1i}+\lambda_ie_{i1}\in X$; and
  \item for every~$x\in X$, we also have $x^*\in X$.
  \end{enumerate}
  We will also say that~$(X,\unorm\cdot)$ is
  \emph{admissible with respect to $\Gamma$}.

\item   We say that a (unitarily invariant) norm~$\|\cdot\|$ on~$\M_n(\bF)$
  is \emph{admissible} if $(\M_n(\bF),\|\cdot\|)$ is admissible
  in~$\M_n(\bF)$.
\end{enumerate}
\end{definition}

\newenvironment{smallbmatrix}{\left[\begin{smallmatrix}}{\end{smallmatrix}\right]}
\let\mat\smallbmatrix\let\emat\endsmallbmatrix

  \begin{example}[$\M_n(\bF)$]
  \label{ex:Mn}
  Let $\unorm\cdot$ be a unitarily invariant norm on~$\M_n(\bF)$ which
  is not a multiple of the Frobenius norm and, in the case
  $(\bF,n)=(\bR,4)$, is not the Ky-Fan $2$-norm.  The
  $\Gamma_\bF$-isometry property holds
  by~\cite[Theorem~4.1]{lt-unitarily90} and~\cite{sourour}
  in the real and complex cases, respectively. Thus
  $(\M_n(\bF),\|\cdot\|)$ is an admissible matrix space.
  
    \end{example}
    
    \begin{example}[$\H_n(\bR)$]
    \label{ex:symmetric}
    Let
  $\unorm\cdot$ be a unitarily invariant norm on~$\M_n(\bR)$ which is
  not a multiple of the Frobenius norm.
    Suppose one of the 
	following conditions holds:
    \begin{enumerate}[(a)]
    \item $n\ne4$, or,
    \item  $\unorm x\ne
      \unorm{\tfrac12(\trace(x))I-x}$ for some $x\in \H_n(\bR)$.
    \end{enumerate}
    Then the subspace $(\H_n(\bR),\|\cdot\|)$ has the $\Gamma_\bR$-isometry property  
    by~\cite[Theorem~6.3]{lt-duality91} and so $(\H_n(\bR),\|\cdot\|)$ is an admissible matrix space in~$\M_n(\bR)$.
    \end{example}
    
    \begin{example}[$\H_n(\bC)$]
    \label{ex:hermitian}
    Let
    $\unorm\cdot$ be a unitarily invariant norm on~$\M_n(\bC)$ which is
    not %
    induced by an inner product.
    Suppose the 
    following conditions hold:
    \begin{enumerate}[(a)]
    \item There does not exist $f:\bR^2\to \bR$ such that $\unorm x=f(|\trace(x)|, \trace(x^2))$ %
      for all $x\in \H_n(\bC)$; and
    \item  $\unorm x\ne
      \unorm{\tfrac2n(\trace(x))I-x}$ for some $x\in  \H_n(\bC)$. 
    \end{enumerate}
      Then the subspace $(\H_n(\bC),\|\cdot\|)$ has the $\Gamma_\bR$-isometry property      by~\cite[Theorem~2]{lt-hermitian90} %
    and so $(\H_n(\bC),\|\cdot\|)$ is an admissible matrix space in~$\M_n(\bC)$.

    In particular, $(\H_n(\bC),\|\cdot\|_{c_p})$ is admissible in~$\M_n(\bC)$ for $n\ge3$ and $1\leq p\leq \infty$ with $p\neq 2$; to verify condition~(a), consider  $x_1=\mat1&1&0\\1&0&1\\0&1&1\emat\oplus 0$ and
$x_2=\mat2&1&0\\1&0&0\\0&0&0\emat\oplus 0$.
\end{example}

\begin{example}[$(\H_2(\bC),\|\cdot\|_{c_p})$]
  Consider $\H_2(\bC)$ with the Schatten $p$-norm where $p\ne2$. 
  Condition~(a) in Example~\ref{ex:hermitian} fails, since in the $2\times
    2$ case the two singular values (and hence also the
    $c_p$-norm) of any symmetric $2\times 2$ matrix~$x\in X$ are
    determined by $|\trace(x)|$ and
    $\trace(x^2)$. Following~\cite[Theorem~2(c)]{lt-hermitian90}, in addition to the
    isometries arising from $\Gamma_\bR$ and multiplication by unitary
    matrices, we must also consider isometries $A\colon X\to X$ which
    preserve the bilinear form $(x,y)\mapsto \trace(xy)$ on $X\times
    X$ %
    and have $A(I)=\pm
    I$. We claim that any such~$A$ must be of the form $A(x)=\pm
    u\,\gamma(x)\,u^*$ for some $u\in \U_2(\bC)$ and $\gamma\in
    \Gamma_\bR$, so we do indeed have the $\Gamma_\bR$-isometry
    property. To see this, we may first negate~$A$ if necessary to ensure that
    $A(I)=I$. Note that $\trace(A(x)^2)=\trace(x^2)$ and
    $|\trace(A(x))|=|\trace(A(x)A(I))|=|\trace(xI)|=|\trace(x)|$. Hence
    $A$ preserves singular values, and moreover if $\trace(x)=0$, then
    $\trace(A(x))=0$. Consider $x=\mat1&0\\0&-1\emat$. The singular
    values of~$x$, and hence also $A(x)$, are $(1,1)$. Composing $A$ with a suitable
    unitary conjugation, we can arrange that $A(x)$ is diagonal with 
    monotonically decreasing diagonal entries; since $A(x)$
    has trace $0$, we have $A(x)=x$.
    The subspace spanned by $I$
    and $x$ is $D$, the space of diagonal matrices in~$X$, and we have
    shown that $A$ acts trivially on~$D$. Hence the subspace~$E=D^\perp$ spanned by
    $y=\mat 0&1\\1&0\emat$ and $z=\mat 0&i\\-i&0\emat$ must have
    $A(E)=E$. We have $A(y)=\mat 0&\alpha\\\overline{\alpha}&0\emat$
    for some $\alpha\in \bT$, and $\trace(A(y)A(z))=\trace(yz)=0$, so
    it follows that $A(z)=\mat 0&\beta\\\overline{\beta}&0\emat$ where
    $\beta\in\{i\alpha,-i\alpha\}$. Conjugating by the diagonal unitary $\mat
    1&0\\0&\overline \alpha\emat$, we may assume that~$A$ fixes~$I$, $x$ and~$y$, and 
    $A(z)=\pm z$. So either $A(z)$ or $A(\transpose z)=-A(z)$ is equal
    to~$z$.  Precomposing with the transpose if necessary,
    we reduce~$A$ to the identity map, verifying the claim above.
    Hence $(\H_2(\bC),\|\cdot\|_{c_p})$ is admissible in~$\M_2(\bC)$ provided $p\ne2$.
  \end{example}

  \begin{remark}\label{rk:admissible}
    These examples show that in particular, $\H_n(\bF)$ and
    $\M_n(\bF)$ are admissible in~$\M_n(\bF)$ with respect to the
    Schatten $p$-norm for any~$n\ge2$ and $1\leq p\leq \infty$ with
    $p\ne2$. Note that the Schatten $2$-norm is not admissible; however, it
    arises from an inner product and so the accompanying graph rigidity 
    follows that of the Euclidean norm. 
     \end{remark}
  
\subsection{Rigid motions}

Recall~\cite{kit-pow,kit-sch} that a \emph{rigid motion} of a normed space $(X,\|\cdot\|)$ is a collection of continuous paths $\alpha=\{\alpha_x:[-1,1]\to X\}_{x\in X}$, with the following properties:
\begin{enumerate}[(a)]
\item
$\alpha_x(0)=x$ for all $x\in X$;
\item
$\alpha_x(t)$ is differentiable at $t=0$ for all $x\in X$; and
\item
$\|\alpha_x(t)-\alpha_y(t)\| = \|x-y\|$ for all $x,y\in X$ and for all $t\in [-1,1]$.
\end{enumerate} 
Note that formally, $\alpha$ is a map $\alpha\colon X\times [-1,1]\to X$, $\alpha(x,t)=\alpha_x(t)$ which satisfies these conditions; we will routinely interchange the notation $\alpha(x,t)$ with $\alpha_x(t)$ where it eases the exposition.
We write~$\R(X,\|\cdot\|)$ for the set of all rigid motions of~$(X,\|\cdot\|)$. 
As we will shortly see, in admissible matrix spaces a rigid motion always has 
a particularly nice form near $t=0$. 

\begin{lemma}
\label{lem:rigidmotion}
Let $(X,\|\cdot\|)$ be a normed space and let $\alpha\in\R(X,\|\cdot\|)$.
Then, \begin{enumerate}[(i)]
\item for each $t\in [-1,1]$ there exists a real-linear
  isometry $A_t:X\to X$ and a vector $c(t)\in X$ such that
  \[\alpha_x(t) = A_t(x)+c(t), \quad \forall\,x\in X.\]
\item the map $c:[-1,1]\to X$ is continuous on $[-1,1]$ and 
  differentiable at $t=0$,
	\item  for every~$x\in X$, the map $A_*(x):[-1,1]\to X$, $t\mapsto A_t(x)$, is continuous on $[-1,1]$ and differentiable at $t=0$, and, 
	\item $A_0=I$ and $c(0) = 0$. 
\end{enumerate}
\end{lemma}

\begin{proof}
By property~(c) of the rigid motion $\alpha$, for every fixed
  $t\in [-1,1]$, the map $x\mapsto \alpha_x(t)$ is an isometry of
  $(X,\unorm{\cdot})$. Since~$X$ is finite dimensional, this isometry is necessarily surjective (see for example \cite[p.~500]{bha-sem}) so this is a real-affine
  map by the Mazur-Ulam theorem. Hence there exists a real-linear isometry $A_t:X\to X$ and $c(t)\in X$ such that
  \[\alpha_x(t) = A_t(x)+c(t), \quad \forall\,x\in X.\]
  Note that~$c(t)=\alpha_0(t)$ is a continuous function of~$t$ (and is
  differentiable at $t=0$), so $A_t(x)=\alpha_x(t)-c(t)$ is also a
  continuous function of~$t$ (and is differentiable at $t=0$), for
  every~$x\in X$. Finally, $c(0) = \alpha_0(0) = 0$ and $A_0(x) = \alpha_x(0)=x$ for every~$x\in X$. 
\end{proof}

 In the proof of the following proposition, for $X\subseteq \M_n(\bF)$ we say that a map~$A\colon X\to X$ is \emph{implemented by
    unitaries} if there exist~$r,s\in \U_n(\bF)$ so that $A(x)=rxs$ for
  every~$x\in X$. 
	
\begin{proposition}\label{prop:rigid}
  Let $(X,\unorm\cdot)$ be an admissible matrix space in~$\M_n(\bF)$.
  For any~$\alpha\in\R(X,\|\cdot\|)$, there is a neighbourhood~$T$
  of~$0$ in $[-1,1]$, and matrices $u_t,w_t\in \U_n(\bF)$ and
  $c(t)\in X$ for each $t\in T$, so that
\begin{enumerate}[(i)]
\item
$\alpha_x(t) = u_txw_t+c(t), \quad \forall\,x\in X,\ t\in T$;

\item
$c(0)=0$ and $u_0=w_0=I$; %

\item the maps $t\mapsto c(t)$ and $t\mapsto u_txw_t$ are both differentiable at $t=0$, for any $x\in X$; and
\item
the maps $t\mapsto u_t$ and $t\mapsto w_t$ are continuous at $t=0$.
\end{enumerate}
\end{proposition}

\begin{proof}
  Let $\alpha\in\R(X,\|\cdot\|)$. Then for each $t\in [-1,1]$ there exists a real-linear isometry $A_t:X\to X$ and vector $c(t)\in X$ as in Lemma \ref{lem:rigidmotion}.  
	Consider the set \[T=\{t\in [-1,1]\colon
  \text{$A_t$ is implemented by unitaries}\}.\]
  Note that $0\in T$ since
  $A_0$ is the identity map on~$X$.
  Let $\Gamma$ be a test set with respect to which $(X,\unorm\cdot)$ is admissible.
  By the $\Gamma$-isometry property, for every~$t\in [-1,1]$, there exist $r_t,s_t\in \U_n(\bF)$ and $\gamma_t\in \Gamma$ so that 
  \begin{equation}\label{eq:form1}A_t(x) = r_t\,\gamma_t(x)\,s_t \quad \forall\,x\in X,
  \end{equation}
  and for~$t\in T$ we may insist that~$\gamma_t=\id$. We can also take~$r_0=s_0=I$.

  For $t\in[-1,1]$, let $\theta_t=\arg (\trace(r_t))$, and define $u_t$, $w_t$ by
  \[ u_t=e^{-i\theta_t}r_t,\quad w_t=e^{i\theta_t} s_t.\] Note that
  $\trace(u_t)\geq 0$ for all $t\in [-1,1]$, and
  $u_0=v_0=I$. Moreover, for each $x\in X$, we have
  $A_t(x)=u_t\,\gamma_t(x)\,w_t$. In particular,
  $\alpha_x(t)=u_txv_t+c(t)$ for every~$x\in X$ and~$t\in T$.
  
  If $u_t$ is not continuous at $t=0$, then there exist $\epsilon>0$ and
  a sequence $t_n\to 0$ so that $\|u_{t_n}-I\|\geq\epsilon$ for
  all~$n\in\bN$.  Since $\U_n(\bF)$ is compact, there is a
  subsequence $(t_{n_k})$ such that $(u_{t_{n_k}})$ and $(w_{t_{n_k}})$
  are both convergent, say to ${u}$ and ${w}$, respectively. 
  Then $u,w\in \U_n(\bF)$ and  since $\gamma_t(I)=I$ for every~$t$, we have
  \[ I=A_0(I)=\lim_{k\to \infty} A_{t_{n_k}}(I)=\lim_{k\to \infty} u_{t_{n_k}}\,\gamma_{t_{n_k}}(I)\,w_{t_{n_k}}= u w,\]
  so $w= u^*$.
  Since the test set~$\Gamma$ is finite, passing to a further subsequence if necessary, we can arrange that
  $\gamma_{t_{n_k}}$ is independent of~$k$, say
  $\gamma_{t_{n_k}}=\gamma$ for all $k\ge1$.
  For every $x\in X$, we have
  \[x=A_0(x)=\lim_{k\to \infty}A_{t_{n_k}}(x) = \lim_{k\to\infty}
  u_{t_{n_k}}\,\gamma_{t_{n_k}}(x)\, w_{t_{n_k}} = u\,\gamma(x)
  \,u^*,\] so $xu=u\gamma(x)$, hence $u\in X^\Gamma=\bF I$ since $X$
  is $\Gamma$-large. Now
  $\trace({u})=\lim_{k\to\infty}\trace(u_{t_{n_k}})\geq 0$, so
  $u=I$ and
  \[ 0= \|u-I\|=\lim_{k\to \infty}\|u_{t_{n_k}}-I\|\geq \epsilon>0,\]
  a contradiction.
  Hence $t\mapsto u_t$ is continuous at~$t=0$, so $t\mapsto w_t=u_t^*
  A_t(I)$ is also continuous at~$t=0$.
  
  Finally, if $T$ is not a neighbourhood of~$0$, then there is
  sequence $t_n\to 0$ with $t_n\in [-1,1]\setminus T$ for all
  $n\ge1$. Passing to an infinite subsequence on which $\gamma_t$ is
  constant, we may assume that $\gamma_{t_n}=\gamma$ does not depend
  on~$n$. Let $x\in X$.  Since $t\mapsto A_t(x)$ is continuous at
  $t=0$ and we know that $u_{t_n}\to I$ and $w_{t_n}\to I$ as $n\to
  \infty$, we have
  \[ x=\lim_{n\to \infty} A_{t_n}(x)=\lim_{n\to
    \infty}u_{t_n}\,\gamma(x)\,w_{t_n}=\gamma(x),\] so $x=\gamma(x)$
  for all $x\in X$. In particular, setting $t=t_1\in [-1,1]\setminus
  T$, we have $A_t(x)=u_t\,\gamma(x)\,w_t = u_t\,x\,w_t$ for all $x\in
  X$, so $t\in T$, a contradiction.
\end{proof}

\subsection{Infinitesimal rigid motions}
A vector field $\eta:X\to X$ of the form $\eta(x)= \alpha_x'(0)$ where $\alpha\in \R(X,\|\cdot\|)$ is referred to as  an \emph{infinitesimal rigid motion} of~$(X,\unorm\cdot)$.  We also say that~$\eta$ is \emph{induced} by the rigid motion~$\alpha$. The collection of all infinitesimal rigid motions of a normed space $(X,\unorm\cdot)$ is a real-linear subspace of $X^X$, denoted $\T(X,\unorm\cdot)$.

\begin{lemma}
\label{lem:triv_aff}
If $(X,\|\cdot\|)$ is a normed space, then every
$\eta\in\T(X,\|\cdot\|)$
is an affine map.  
\end{lemma}

\begin{proof}
Suppose~$\eta$ is induced by $\alpha\in \R(X,\|\cdot\|)$. For $x\in X$ and $t\in [-1,1]$, write $\alpha_x(t)=A_t(x)+c(t)$ where the real-linear maps $A_t:X\to X$ and vectors $c(t)\in X$ are as in Lemma~\ref{lem:rigidmotion}.
Then $\eta(x) = \alpha_x'(0) = B(x) + c'(0)$ where $B\colon X\to X$ is the real-linear map given by $B(x)=\frac{d}{dt}A_t(x)|_{t=0}$.
\end{proof}

From the viewpoint of infinitesimal rigidity theory, which we consider in Section~\ref{GraphRigidity}, infinitesimal rigid motions yield trivial deformations of a framework since they arise from a global deformation of~$X$. We will now identify these in our context.

\begin{theorem}\label{thm:skew}
  Let $\unorm\cdot$ be a unitarily invariant norm on~$\M_n(\bF)$, and
  suppose that $(X,\unorm\cdot)$ is an admissible matrix space. If
  $\eta\in \T(X,\unorm\cdot)$, then there exist unique matrices
  $a,b,c\in \M_n(\bF)$ with $a\in \Skew_n(\bF)$, $b\in \Skew_n^0(\bF)$ and~$c\in X$ so that
  \[\eta(x)=ax+xb+c, \quad \forall\, x\in X.\]  %
\end{theorem}
\begin{proof}
Choose some $\alpha\in \R(X,\|\cdot\|)$ which induces~$\eta$ and consider a neighbourhood~$T$ of~$0$ and maps $u,w: T\to \U_n(\bF)$, $u(t)=u_t$ and $w(t)=w_t$ and $c:T\to X$ as in Proposition~\ref{prop:rigid}.
Note in particular that these maps are continuous at $t=0$, with $u_0=w_0=I$ and $c(0)=0$, and  for all $x\in X$, the restriction of~$\alpha_x$ to~$T$ is given by
  \[ \alpha_x(t)=u_txw_t+c(t)\]
  and this restriction is differentiable at $t=0$. 

Suppose first that $c(t)=0$ for all $t\in T$.  

Consider the map
\[\delta_r : X\to \M_n(\bF),\qquad
  \delta_r(x)=\alpha_x'(0)-\alpha_I'(0)x.\] Note that for each
$x\in X$, we have
\[\delta_r(x)= \lim_{t\to 0} \frac{u_txw_t-x-(u_tw_t-I)x}t =
    \lim_{t\to 0} \tfrac1t u_t [x,w_t].\]
Since $u_t^*\to I$ as $t\to 0$, we have		
\[ \delta_r(x)= \lim_{t\to 0} \tfrac1t u_t^*u_t [x,w_t] = \lim_{t\to 0} \tfrac1t [x,w_t]. \]
  Observe that if $s\in \M_n(\bF)$ has $s_{11}=0$, then for any $\lambda\in\bF$ and
  $1\leq i,j\leq n$, the $(i,j)$ entry of $s$ is given by
  \[ s_{ij}=
    \begin{cases}
      [e_{ii},s]_{ij}&\text{ if $i\ne j$},\\
      [e_{1i}+\lambda e_{i1},s]_{1i}&\text{ if $i=j$}.
    \end{cases}
  \]
  For $0\ne t\in T$, let $b_t=t^{-1}(w_t-(w_t)_{11}I)$, so that
  $\delta_r(x)=\lim_{t\to 0}[x,b_t]$ for $x\in X$ and the $(1,1)$
  entry of~$b_t$ is~$0$. Since~$X$ is admissible, the preceding
  observation shows that $b_t$ is entrywise convergent, say $b_t\to b$
  as $t\to 0$, hence $\delta_r(x)=[x,b]$ for each~$x\in X$.  Note that the $(1,1)$ entry of~$b$ is~$0$. Let
  $a=\alpha_I'(0)-b$; then
  \[ \alpha_x'(0)=\alpha_I'(0)x+\delta_r(x) = ax+xb,\quad x\in X. \]
  
  For each~$x\in X$, consider the map \[\beta_x:T\to \M_n(\bF),\quad \beta_x(t)=w_txu_t.\] 
  For $t\in T$, we have
  \begin{equation*}
    \beta_x(t)-\beta_x(0)= w_txu_t-x=w_t(x-w_t^*xu_t^*)u_t=w_t(x^*-u_tx^*w_t)^*u_t,
  \end{equation*}
  so by the continuity of $u_t$ and $w_t$ at $t=0$, we have
  \begin{align*} \beta_x'(0)&=\lim_{t\to 0}\frac{\beta_x(t)-\beta_x(0)}t=\lim_{t\to 0}w_t\left(\frac{x^*-u_tx^*w_t}t\right)^*u_t
    \\&= -\alpha_{x^*}'(0)^*=
    -b^*x-xa^*.
  \end{align*}

  Now 
  \begin{align*}
    xb-bx&=\delta_r(x)=\lim_{t\to 0}\tfrac1t[x,w_t]=\lim_{t\to 0}\tfrac1t [x,w_t]u_t
           \\ &= \lim_{t\to 0}\tfrac1t(xw_t-w_tx)u_t
                = \lim_{t\to 0}\tfrac1tx(w_tu_t-I)-\tfrac1t(w_txu_t-x)\\
    &=x\beta_I'(0)-\beta_x'(0) = -x(a^*+b^*)+(b^*x+xa^*)
      \\&= b^*x-xb^*
  \end{align*}
  so $x(b+b^*)=(b+b^*)x$ for all $x\in X$, so $b+b^*\in
  X'=\bF I$. Since $b_{11}=0$, we have $b+b^*=0$.
	
  Define $\delta_\ell(x)=\alpha'_x(0)-x\alpha'_I(0)$. We know that
  $\alpha'_x(0)=ax+xb$, so $\delta_\ell(x)=ax+xb-x(a+b)=[a,x]$. A
  similar computation to the one above for $\delta_r$ yields
  $\delta_\ell(x)=\beta'_I(0)x-\beta'_x(0)$. It follows that
  $a+a^*\in \bF I$, and hence that $a+a^*=\lambda I$ for some
  $\lambda\in \bR$.
  
  \let\phi\varphi
  Now consider the maps $\phi_+,\phi_-\colon \M_n(\bF)\to \bR$ given by the one-sided limits
  \[ \phi_{\pm}(x)=\lim_{t\to
    0^{\pm}}\frac{\unorm{I+tx}-\unorm{I}}t,\qquad x\in X.\]
  These limits are well defined (see, for
  example,~\cite[Theorem~23.1]{rockafellar}); moreover, $\phi_+$ is sub-additive and
  $\phi_-$ is super-additive, and $\phi_{\pm}(\alpha I)=\alpha \|I\|$
  for any~$\alpha\in\bR$. Note that
  \[ \alpha_I'(0)=a+b\qand \unorm{I}=\unorm{\alpha_I(t)}\text{ for any~$t\in \bR$}.\]
  It follows that $\phi_{\pm}(a+b)=0$,  since
  \begin{align*}
    \left|\frac{\unorm{I+t(a+b)}-\unorm{I}}t\right|&=
    \left|\frac{\unorm{\alpha_I(0)+t\alpha_I'(0)}-\unorm{\alpha_I(t)}}t\right|\\&\leq
    \left\unorm{\frac{\alpha_I(t)-\alpha_I(0)}t - \alpha_I'(0)\right}\to 0 \text{ as $t\to 0$.}
  \end{align*}
  The conjugate transpose is isometric for the unitarily
  invariant norm~$\unorm\cdot$, so
  $\phi_{\pm}(x^*)=\phi_{\pm}(x)$ for any~$x\in X$. Hence
  $\phi_{\pm}(a^*+b^*)=0$. Since $b+b^*=0$, we have 
  \[\lambda I=a+a^*=a+b+a^*+b^*.\]
  Applying $\phi_+$ and using sub-additivity, we obtain 
  \[ \lambda\|I\|=\phi_+(a+b+a^*+b^*)\leq \phi_+(a+b)+\phi_+(a^*+b^*)=0.\]
  Applying $\phi_-$ similarly, we obtain the converse inequality, so $\lambda=0$. 
  Thus
  $a$ and $b$ are skew-hermitian, with
  $b\in \Skew_n^0(\bF)$.

  For uniqueness, if $(a',b')\in \Skew_n(\bF)\times \Skew_n^0(\bF)$
  with $ax+xb=a'x+xb'$ for every~$x\in X$, then $a''x+xb''=0$ where
  $a''=a-a'$ and $b''=b-b'$. Setting $x=I$ gives $b''=-a''$, so
  $b''\in \Skew_n^0(\bF)\cap X'=\Skew_n^0(\bF)\cap \bF I =
  \{0\}$, so $a''=b''=0$.

  Finally, if $c(t)$ is not identically zero then applying the above argument to the rigid motion obtained by replacing $\alpha_x(t)$ with $\alpha_x(t)-c(t)$ for each $x\in X$, we obtain  $\alpha_x'(0)=ax+xb+c$ for some unique $(a,b)\in \Skew_n(\bF)\times \Skew_n^0(\bF)$ and where $c=c'(0)\in X$.
\end{proof}

In the case of admissible spaces of the form $(\H_n(\bF),\|\cdot\|)$ we obtain the following refinement of Theorem \ref{thm:skew}.

\begin{corollary}
\label{cor:skew}
Let $\unorm\cdot$ be a unitarily invariant norm on~$\M_n(\bF)$. If $(\H_n(\bF),\unorm\cdot)$ is admissible and  $\eta\in \T(\H_n(\bF),\unorm\cdot)$, 
then there exist 
 unique matrices $a\in \Skew_n^0(\bF)$  and $c\in \H_n(\bF)$ so that
  \[\alpha_x'(0)=ax-xa+c, \quad \forall\, x\in \H_n(\bF).\] 
\end{corollary}

\begin{proof}
  Applying Theorem~\ref{thm:skew} with $X=\H_n(\bF)$ to obtain $a\in \Skew_n(\bF)$, $b\in \Skew_n^0(\bF)$ and
  $c\in \H_n(\bF)$, we observe that \[\alpha'_I(0)-c=a+b\in
  \H_n(\bF)\cap \Skew_n(\bF)=\{0\},\] so $b=-a$.
\end{proof}

\subsection{The dimension of $\T(X,\|\cdot\|)$}\label{subsec:dim-T}
Let $(X,\|\cdot\|)$ be an admissible matrix space.
By Theorem~\ref{thm:skew},  there is a well-defined map
\[\Psi_X:\T(X,\|\cdot\|)\to \Skew_n(\bF)\oplus \Skew_n^0(\bF)\oplus\M_n(\bF),\]
 with the property that $\Psi_X(\eta)=(a,b,c)$ if and only if $\eta(x)=ax+xb+c$ for all $x\in X$.

\begin{lemma}
\label{lem:iso}
The map $\Psi_X$ is injective and linear. Moreover, if $X=\M_n(\bF)$ then  $\Psi_{X}$ is a  linear isomorphism.
\end{lemma}

\begin{proof}
That   $\Psi_X$ is injective and linear is a routine verification.
Suppose $X=\M_n(\bF)$. Then it only remains to prove surjectivity.
Let $(a,b,c)$ be in the codomain of $\Psi$, and for each $x\in \M_n(\bF)$ define
\[\alpha_x:[-1,1]\to \M_n(\bF), \quad \alpha_x(t) = e^{ta}xe^{tb}+tc.\]
Since $a$ and $b$ are skew-hermitian, $e^{ta}$ and $e^{tb}$ are unitary for every~$t\in \bR$, so 
$\{\alpha_x:[-1,1]\to \M_n(\bF)\}_{x\in \M_n(\bF)}$ is a rigid motion of $(\M_n(\bF),\|\cdot\|)$.
Differentiating, we see that the induced infinitesimal rigid motion is the vector field
\[\eta:X\to X,\,\,\,\,\,x\mapsto ax+xb+c.\]
Thus $\Psi(\eta)=(a,b,c)$ and so $\Psi$ is surjective.
\end{proof}

Here and below, we write $\dim Z$ for the real-linear dimension of a real-linear vector subspace~$Z$ of~$\M_n(\bF)$
or $\M_n(\bF)^{\M_n(\bF)}$.

\begin{proposition}
\label{prop:dim1}
If $(\M_n(\bF),\|\cdot\|)$ is an admissible matrix space, then
\[ \dim \T(\M_n(\bF),\|\cdot\|)=
\begin{cases}
  2n^2-n&\text{if $\bF=\bR$,}\\
  4n^2-1&\text{if $\bF=\bC$.}
\end{cases}\]
\end{proposition}

\begin{proof}
By Lemma \ref{lem:iso}, $\Psi_{\M_n(\bF)}$ is a linear isomorphism. If $\bF=\bR$, then
 \begin{align*}
  \dim \T(\M_n(\bR),\|\cdot\|) &= \dim (\Skew_n(\bR)\oplus \Skew_n(\bR)\oplus \M_n(\bR)) \\
 &= \frac{n(n-1)}{2}+\frac{n(n-1)}{2} +n^2 \\
 &= 2n^2-n.
 \end{align*}
If $\bF=\bC$, then
 \begin{align*}
 \dim \T(\M_n(\bC),\|\cdot\|) &=  \dim (\Skew_n(\bC)\oplus \Skew_n^0(\bC)\oplus \M_n(\bC)) \\
 &= n^2-1+n^2+2n^2 \\
 &= 4n^2-1.\qedhere
 \end{align*}
\end{proof}

We now compute the dimension of the space of
infinitesimal rigid motions for admissible matrix
spaces of the form~$(\H_n(\bF),\|\cdot\|)$. 

\begin{lemma}
\label{lem:range}
The range of $\Psi_{\H_n(\bF)}$ is 
\[\range \Psi_{\H_n(\bF)} = \{(a,-a,c)\colon (a,c)\in \Skew_n^0(\bF)\oplus \H_n(\bF)\}.\]
\end{lemma}

\begin{proof}
By Corollary \ref{cor:skew}, if $(a,b,c)$ is an element of the range of  $\Psi_{\H_n(\bF)}$ then $b=-a$ and $c\in \H_n(\bF)$.
For the reverse inclusion,
let $a\in \Skew_n^0(\bF)$, let $c\in\H_n(\bF)$, and for each $x\in \H_n(\bF)$ define
\[\alpha_x:[-1,1]\to \H_n(\bF), \quad \alpha_x(t) = e^{ta}xe^{-ta}+tc.\]
Then $\{\alpha_x:[-1,1]\to \H_n(\bF)\}_{x\in \H_n(\bF)}$ is a rigid motion of $(\H_n(\bF),\|\cdot\|)$.
The induced infinitesimal rigid motion is the vector field
\[\eta:\H_n(\bF)\to \H_n(\bF),\,\,\,\,\,x\mapsto ax-xa+c.\]
Thus $\Psi_{\H_n(\bF)}(\eta)=(a,-a,c)$ and so $(a,-a,c)$ is contained in the range of $\Psi_{\H_n(\bF)}$.
\end{proof}

\begin{proposition}
\label{prop:dim2}
If $(\H_n(\bF),\|\cdot\|)$ is an admissible matrix space, then
\[ \dim \T(\H_n(\bF),\|\cdot\|)=
\begin{cases}
  n^2&\text{if $\bF=\bR$,}\\
  2n^2-1&\text{if $\bF=\bC$.}
\end{cases}\]
\end{proposition}

\begin{proof}
By Lemma \ref{lem:iso}, $\Psi_{\H_n(\bF)}$ is a linear isomorphism onto its range. Thus by
Lemma \ref{lem:range} we have 
$\dim(\T(\H_n(\bF),\|\cdot\|))=
\dim (\Skew_n^0(\bF)\oplus \H_n(\bF))$, which gives the advertised values.
\end{proof}

\section{Infinitesimal rigidity for admissible matrix spaces}
\label{GraphRigidity}

In this section we develop infinitesimal rigidity theory for admissible matrix spaces. Our primary goal is to obtain necessary counting conditions for graphs which admit an infinitesimally rigid placement in a given admissible matrix space. This is achieved in Theorem~\ref{thm:maxwell}, where we provide analogues of the Maxwell counting criteria for Euclidean bar-joint frameworks~\cite{maxwell}, and in Theorem~\ref{thm:tight}, where we show that minimally rigid graphs belong to the  matroidal class of $(k,l)$-sparse graphs for suitable $k$ and $l$ (see~\cite{lee-str}).  

Throughout this section $X$ will be a finite dimensional real linear space and $G=(V,E)$ will be a finite simple graph. A \emph{bar-joint framework} in  $X$ is a pair  $(G,p)$  consisting of a graph $G$ and a map $p:V\to X$, $v\mapsto p_v$, called a \emph{placement} of~$G$ in~$X$,  with the property that $p_v\not= p_w$ for all $vw\in E$. A \emph{subframework} of $(G,p)$ is a bar-joint framework $(H,p_H)$ with $H=(V(H),E(H))$ a subgraph of $G$ and $p_H(v)=p(v)$ for all $v\in V(H)$. 

\subsection{Support functionals}
Recall that if~$\|\cdot\|$ is a norm on $X$, then a \emph{support functional} for a unit vector $x_0\in X$ is a linear functional $f:X\to \bR$ with $\|f\|:=\sup\{|f(x)|\colon x\in X,\|x\|=1\}\leq 1$, and $f(x_0) =1$. The norm~$\|\cdot\|$ is said to be \emph{smooth} at $x\in X\setminus\{0\}$ if there exists exactly one support functional at $\frac{x}{\|x\|}$, and we say that~$\|\cdot\|$ is \emph{smooth} if it is smooth at every $x\in X\setminus\{0\}$.

We will require the following facts (for details see~\cite[Section~2]{kit-sch}).

\begin{lemma}
\label{lem:smooth}
Let $(G,p)$ be a bar-joint framework in a normed linear space $(X,\|\cdot\|)$, let $vw\in E$ and let $p_0=\frac{p_v-p_w}{\|p_v-p_w\|}$.
\begin{enumerate}[(i)]
\item 
The norm~$\|\cdot\|$ is smooth at $p_v-p_w$ if and only if the limit 
\begin{equation}
\label{eq:support-functional-limit}
\varphi_{v,w}(x) := \lim_{t\to 0}\frac{1}{t}(\|p_0+tx\|-\|p_0\|)
\end{equation}
exists for all $x\in X$.
\item
If the norm is smooth at $p_v-p_w$, then the map $\varphi_{v,w}:X\to \bR$ is the unique support functional for~$p_0$.
\end{enumerate}
\end{lemma}

Recall from the introduction that every unitarily invariant norm on $\M_n(\bF)$ arises from a symmetric norm on $\bR^n$ and that $\sigma(x)\in\bR^n$ denotes the vector of singular values, arranged in decreasing order, for a matrix $x\in\M_n(\bF)$.

\begin{lemma}\label{lem:smoothness}
  Let~$\|\cdot\|$ be a unitarily invariant norm on~$\M_n(\bF)$, with
  corresponding symmetric norm~$\|\cdot\|_s$ on~$\bR^n$, and let  $x\in \M_n(\bF)$. Then $\|\cdot\|$ is smooth at $x$
  if and only if~$\|\cdot\|_s$ is smooth at $\sigma(x)$.
\end{lemma}

\begin{proof}
The result follows from~\cite[Theorem~2]{watson}.
\end{proof}

Support functionals for the Schatten $p$-norms are described in \cite{arazy}. We apply these results below to characterise the support functionals $\varphi_{v,w}$.

\begin{example}
  Let $1\leq q\leq \infty$ and let $(G,p)$ be a bar-joint framework in $(\M_n(\bF),\|\cdot\|_{c_q})$. Let $vw\in E$, suppose the norm is smooth at $p_v-p_w$ and let $p_0=\frac{p_v-p_w}{\|p_v-p_w\|_{c_q}}$.
\begin{enumerate}[(a)]
\item If $q<\infty$, then for all $x\in \M_n(\bF)$, 
\[\varphi_{v,w}(x) = \tr(x|p_0|^{q-1}u^*)\]
where $p_0=u|p_0|$
is the polar decomposition of~$p_0$.

\item If $q=\infty$, then by Lemma \ref{lem:smoothness}, the largest singular value of the matrix $p_0$ has multiplicity one. Thus $p_0$ attains its norm at a unit vector
  $\zeta\in \bF^n$ which is unique (up to scalar multiples). It follows that for all $x\in \M_n(\bF)$, we have
  \[\varphi_{v,w}(x) = \langle x \zeta, p_0 \zeta\rangle\]
  where $\langle\cdot,\cdot\rangle$ is the usual Euclidean inner
  product on~$\bF^n$.
\end{enumerate}
\end{example}

\subsection{Well-positioned frameworks}
\label{subsection:well-pos-support} 

A bar-joint framework $(G,p)$ is said to be \emph{well-positioned} in $(X,\|\cdot\|)$  if 
the norm~$\|\cdot\|$ is smooth at $p_v-p_w$ for every edge $vw\in E$. 

The following criteria apply to well-positioned bar-joint frameworks in the case of Schatten $p$-norms.

\begin{proposition}
  Let~$1\leq q\leq \infty$ with $q\ne2$, and suppose that~$(X,\|\cdot\|_{c_q})$ is an admissible matrix space in~$\M_n(\bF)$. 
  Let $(G,p)$ be a bar-joint framework in $(X,\|\cdot\|_{c_q})$. 
\begin{enumerate}[(i)]
\item\label{not12inf} If $q\not\in \{1,\infty\}$, then $(G,p)$ is well-positioned.
\item If~$q=1$ and $p_v-p_w$ is invertible for all $vw\in E$, then
  $(G,p)$ is well-positioned. For $X=\M_n(\bF)$, the converse also
  holds.
\item If~$q=\infty$ and $\sigma_1(p_v-p_w)>\sigma_2(p_v-p_w)$ for all
  $vw\in E$, then $(G,p)$ is well-positioned. For $X=\M_n(\bF)$,
  the converse also holds.
\end{enumerate}
\end{proposition}

\begin{proof}
  Observe first that if~$(G,p)$ is well-positioned 
	in~$(\M_n(\bF),\|\cdot\|_{c_q})$, 
	then $(G,p)$ is necessarily well-positioned in~$(X,\|\cdot\|_{c_q})$. Hence it suffices to give a proof in the case $X=\M_n(\bF)$. Recall that the
  $\ell_q$ norm on~$\bR^n$ is smooth at the following vectors:
\begin{enumerate}[(i)]
\item at every non-zero vector in~$\bR^n$ if~$q\not\in \{1,\infty\}$;
\item at every vector with every entry non-zero if $q=1$; and
\item at every vector~$\sigma=(\sigma_1,\dots,\sigma_n)$ so that
  $\max_{1\leq i\leq n}|\sigma_i|$ is attained at precisely one $i\in
  \{1,2,\dots,n\}$, if $q=\infty$.
\end{enumerate}
It now suffices to apply
  Lemma~\ref{lem:smoothness}. 
\end{proof}

\subsection{The rigidity map}
As in~\cite{kit-sch}, we consider the \emph{rigidity map} $f_G$, given by
\[f_G:X^{V}\to \bR^{E}, \qquad (x_v)_{v\in V}\mapsto
(\|x_v-x_w\|)_{vw\in E}.\] 
If the rigidity map is differentiable at $p\in X^V$, then \[df_G(p):X^{V}\to \bR^{E},\] is the differential of~$f_G$ at~$p$. Here we equip~$X^V$ %
with the norm topology.

We will require the following results.

\begin{lemma}{\cite[Proposition 6]{kit-sch}}
\label{lem:flexcondition}
Let $(G,p)$ be a bar-joint framework in a normed linear space $(X,\|\cdot\|)$.
\begin{enumerate}[(i)]
\item
$(G,p)$  is well-positioned in $(X,\|\cdot\|)$  if and only if the rigidity map $f_G$ is differentiable at $p$. 
\item
If $(G,p)$ is well-positioned in $(X,\|\cdot\|)$ 
then the differential of the rigidity map is given by
\[df_G(p):X^{V}\to \bR^{E}, \,\,\,\,\,\,\,
(z_v)_{v\in V} \mapsto (\varphi_{v,w}(z_v-z_w))_{vw\in E}.\]
\end{enumerate}
\end{lemma}

An \emph{infinitesimal flex} of a bar-joint framework $(G,p)$ is a vector $z\in X^V$ such that
\[\lim_{t\to 0} \frac{1}{t}(f_G(p+tz)- f_G(p))=0.\]
The collection of all infinitesimal flexes of $(G,p)$ is denoted $\F(G,p)$. 
Note that, by Lemma \ref{lem:flexcondition}, if $(G,p)$ is well-positioned then $\F(G,p)=\ker df_G(p)$. 

\subsection{Full sets}

 Given a normed space~$(X,\|\cdot\|)$, and  a non-empty subset $S\subseteq X$, consider the restriction map,
\[\rho_{S}: \T(X,\|\cdot\|)\to X^S, \quad \eta\mapsto (\eta(x))_{x\in S}.\] 

\begin{definition}\label{def:full}
  A non-empty subset $S\subseteq X$ is \emph{full} in $(X,\|\cdot\|)$ if the restriction map $\rho_{S}$ is injective; that is, if $S$ is a separating set for~$\T(X,\|\cdot\|)$.
\end{definition}

Recall that  $S$ is said to have  \emph{full affine span} in $X$  if $[S]=X$, where 
$[S]$ is the affine span of~$S$, namely the linear span of $\{s-s_0\colon s\in S\}$ where
$s_0$ is any fixed vector in $S$. (Note that %
$[S]$ is
independent of the choice of~$s_0$). 

\begin{lemma}
\label{lem:full3}
Let $(X,\|\cdot\|)$ be a normed space   and
let $\emptyset\ne S\subseteq X$.  
If~$S$ has full affine span in~$X$, then $S$ is full in $(X,\|\cdot\|)$. 
\end{lemma}

\begin{proof}
  Let $\eta\in\T(X,\|\cdot\|)$ and suppose $\rho_S(\eta)=0$.
	By Lemma \ref{lem:triv_aff},
	$\eta$ is an affine map and so $\eta(X)=\eta([S])=0$.
 \end{proof}

 \begin{remark}
   	Note that full affine span is not strictly necessary for $S$ to be full in a normed space $(X,\|\cdot\|)$. 
 	For example, if $[S]$ is the set of upper triangular $n\times n$ matrices then it is not difficult to see that %
         $S$ is full in $(\M_n(\bF),\|\cdot\|)$ for any admissible norm. %
 	\end{remark}

\begin{definition}
We say that a bar-joint framework~$(G,p)$ in a normed space~$(X,\|\cdot\|)$  is, 
	\begin{enumerate}[(a)]
	\item \emph{full} if $\{p_v\colon v\in V\}$ is full in $(X,\|\cdot\|)$. 
	\item \emph{completely full} if $(G,p)$, and every subframework $(H,p_H)$ of $(G,p)$ with $|V(H)|\geq 2\dim(X)$, is full in $(X,\|\cdot\|)$.
	\end{enumerate}
\end{definition}

\begin{remark}
 We remark that the property of being completely full, which will be required in Theorem \ref{thm:tight},  is satisfied by almost all bar-joint frameworks.
Indeed, if $(G,p)$ is a bar-joint framework in $X$ and $S=\{p_v:v\in V\}$ is in general position in~$X$, then every subset of $S$ containing at least $\dim(X)+1$ points has full affine span in $X$. Thus, by Lemma \ref{lem:full3}, $(G,p)$ is completely full in $(X,\|\cdot\|)$, for all norms on $X$. 
\end{remark}
		
\subsection{$k(X)$ and $l(X)$ values}

For $X\in \{\M_n(\bF),\H_n(\bF)\}$, we define natural numbers $k(X)$ and $l(X)$ according to the formulae in Table~\ref{table:kl1}. Note that $k(X)=\dim X$ and by Propositions~\ref{prop:dim1} and~\ref{prop:dim2}, we have $l(X)=\dim(\T(X,\|\cdot\|))$ for any admissible norm $\|\cdot\|$ on~$X$. For ease of reference, the cases $n=2$ and $n=3$ are listed in Table~\ref{table:kl2}.

\begin{table}[ht]
\centering
{\renewcommand{\arraystretch}{1.4}%
\begin{tabular}{|c||c|c| }
  \hline
  $X$ & $k(X)$ & $l(X)$ \\ [1ex] \hline 
  $\H_n(\bR)$ & $\tfrac12n(n+1)$ & $n^2$ \\ [1ex] \hline
	$\M_n(\bR)$ & $n^2$ & $2n^2-n$ \\  [1ex] \hline
  $\H_n(\bC)$ & $n^2$ & $2n^2-1$ \\ [1ex] \hline
  $\M_n(\bC)$ & $2n^2$ & $4n^2-1$ \\ \hline
\end{tabular}}
\vspace{4mm}
\caption{$k$ and $l$ values for admissible matrix spaces.} %
\label{table:kl1} %
\end{table}

\begin{table}[ht]
\centering
{\renewcommand{\arraystretch}{1.25}%
\begin{tabular}{|c||c|c| c c |c||c|c| }
  \cline{1-3} \cline{6-8}
  $X$ & $k(X)$ & $l(X)$ & & & $X$ & $k(X)$ & $l(X)$ \\ [1ex] \cline{1-3} \cline{1-3} \cline{6-8} \cline{6-8} 
  $\H_2(\bR)$ & $3$ & $4$ & & & $\H_3(\bR)$ & $6$ & $9$ \\ [1ex] \cline{1-3} \cline{6-8}
	$\M_2(\bR)$ & $4$ & $6$ & & & $\M_3(\bR)$ & $9$ & $15$ \\  [1ex] \cline{1-3} \cline{6-8}
  $\H_2(\bC)$ & $4$ & $7$ & & & $\H_3(\bC)$ & $9$ & $17$\\ [1ex] \cline{1-3} \cline{6-8}
	$\M_2(\bC)$ & $8$ & $15$ & & & $\M_3(\bC)$ & $18$ & $35$\\  \cline{1-3} \cline{6-8}
\end{tabular}}
\vspace{4mm}
\caption{$k$ and $l$ values for admissible matrix spaces when $n=2$ and $n=3$.} 
\label{table:kl2} %
\end{table}

We will require the following result. As usual, we take $n\in\bN$ with $n\ge2$, and for~$m\in \bN$ we write~$K_m$ for the complete graph on~$m$ vertices. 

\begin{lemma}
\label{lem:comb}
Let~$X\in \{\M_n(\bF),\H_n(\bF)\}$, let $(k,l)=(k(X),l(X))$. Consider
$m\in \bN$.
\begin{enumerate}[(i)]
\item $|E(K_m)|\leq km-l$ if and only if $m\in\{2,\ldots,2k-1\}$. 
\item $|E(K_m)|= km-l$ if and only if $\bF=\bC$ and $m=\{2, 2k-1\}$.
\end{enumerate}
\end{lemma}

\begin{proof}
Consider the quadratic function~$f\colon\bR\to \bR$ given by
\[f(t)= \frac{1}{2}(t^2-(2k+1)t+2l).\]
It is easy to see that $f(m)=|E(K_m)|-(km-l)$, and $f(1)=f(2k)=l-k>0$. Moreover,
\[f(2)=f(2k-1)=l+1-2k=
  \begin{cases}
    -(n-1)&\text{if $\bF=\bR$,}\\
    0&\text{if $\bF=\bC$}
  \end{cases}\]
so $f(2)=f(2k-1)\leq0$, with equality if and only if $\bF=\bC$. 
The claims follow immediately.
\end{proof}

\subsection{Trivial infinitesimal flexes}

Given a bar-joint framework~$(G,p)$, we define
\[\T(G,p)=\{ \zeta\colon V\to X\mid  \zeta=\eta\circ p\text{ for some $\eta\in\T(X,\|\cdot\|)$}\}\subseteq X^V.\]
Note that  $\T(G,p)$ is a subspace of $\F(G,p)$, the space of infinitesimal flexes of $(G,p)$ (see \cite[Lemma 2.3]{kit-pow}).
The elements of $\T(G,p)$ are referred to as the {\em
  trivial infinitesimal flexes} of  $(G,p)$.
  
\begin{example}
\label{example:triv1}
Suppose $(G,p)$ is a bar-joint framework in an admissible matrix space $(X,\|\cdot\|)$.
If $X=\M_n(\bF)$, then by Lemma~\ref{lem:iso} we have
\[ \T(G,p)=\{(ap_v+p_vb+c)_{v\in V}\colon a\in \Skew_n(\bF),\ b\in \Skew_n^0(\bF),\ c\in \M_n(\bF) \},\]
and if $X=\H_n(\bF)$, then by Lemma~\ref{lem:range} we have
  \[\T(G,p)=\{(ap_v-p_va+c)_{v\in V}\colon a\in \Skew_n^0(\bF),\ c\in \H_n(\bF) \}.\] 
\end{example}

\begin{lemma}
\label{lem:triv2}
If $(G,p)$ is a full  bar-joint framework in a normed linear space $(X,\|\cdot\|)$, then
\[ \dim \T(G,p)=\dim \T(X,\|\cdot\|).\] In particular, if $X\in \{\M_n(\bF),\H_n(\bF)\}$ and $(X,\|\cdot\|)$ is an admissible matrix space, then $\dim\T(G,p)=l(X)$.
\end{lemma}

\begin{proof}
Observe that the linear map 
\[\rho_{(G,p)}:\T(X,\|\cdot\|)\to X^V, \,\,\,\,\,\,\, \eta\mapsto(\eta(p_v))_{v\in V},\] has range~$\T(G,p)$.
Since $\{p_v:v\in V\}$ is full in $(X,\|\cdot\|)$, $\rho_{(G,p)}$ is  also injective. Thus, $\dim \T(G,p)=\dim \T(X,\|\cdot\|)$. 
\end{proof}

\subsection{Infinitesimal rigidity}
A bar-joint framework $(G,p)$ is \emph{infinitesimally rigid} if every infinitesimal flex of $(G,p)$ is trivial (i.e., if $\F(G,p)=\T(G,p)$); otherwise, we say that~$(G,p)$ is \emph{infinitesimally flexible}.
A framework $(G,p)$ is said to be \emph{minimally infinitesimally rigid} if it is infinitesimally rigid and every subframework obtained by removing an edge from $G$ is infinitesimally flexible.

The following results are analogous to Maxwell's counting criteria for bar-joint frameworks in Euclidean space \cite{maxwell}.

\begin{theorem}
\label{thm:maxwell}
Let $(G,p)$ be a full and well-positioned bar-joint framework in an admissible matrix space $(X,\|\cdot\|)$, where  
$X\in \{\M_n(\bF),\H_n(\bF)\}$, and let  $(k,l)=(k(X),l(X))$. 
\begin{enumerate}[(i)]
\item If $(G,p)$ is infinitesimally rigid, then $|E|\geq k|V|-l$.
\item If $(G,p)$ is minimally infinitesimally rigid, then
$|E|=k|V|-l$.
\item
If $(G,p)$ is minimally infinitesimally rigid and $(H,p_H)$ is a full subframework of $(G,p)$, then $|E(H)|\leq k|V(H)|-l$.
\end{enumerate}
\end{theorem}

\begin{proof}
Apply \cite[Theorem 10]{kit-sch} and Lemma \ref{lem:triv2}.
\end{proof}

Let $k,l\in\bN$ with $l\in\{0,\ldots,2k-1\}$.
As is standard in combinatorial rigidity theory, a graph $G=(V,E)$ is said to be \emph{$(k,l)$-sparse} if every subgraph $H=(V(H),E(H))$ with  $|V(H)|\geq 2$ has at most  $k|V(H)|-l$ edges.
If in addition $|E| = k|V|-l$, then $G$ is said to be \emph{$(k,l)$-tight}.

\begin{theorem}
\label{thm:tight}
Let~$\|\cdot\|$ be an admissible norm on~$X\in \{\M_n(\bF),\H_n(\bF)\}$, and let $(k,l)=(k(X),l(X))$. Let $(G,p)$ be a completely full and well-positioned bar-joint framework in $(X,\|\cdot\|)$. If $(G,p)$ is minimally infinitesimally rigid,  then 
$G$ is $(k,l)$-tight.
\end{theorem}

\begin{proof}
By Theorem~\ref{thm:maxwell}$(ii)$, $|E|=k|V|-l$. %
Let $H$ be a subgraph of $G$ with $m\geq 2$ vertices.
If $m\geq 2k$ then, since $(G,p)$ is completely full, the subframework $(H,p_H)$ is full in $(X,\unorm{\cdot})$. Thus, by Theorem~\ref{thm:maxwell}$(iii)$, $|E(H)|\leq k|V(H)|-l$. %
If $2\leq  m\leq 2k-1$, then by Lemma \ref{lem:comb},  $|E(H)|\leq |E(K_m)|\leq k|V(H)|-l$. 
Thus $G$ is $(k,l)$-tight.
\end{proof}

\begin{remark}
The  $(k,l)$-sparsity of a multi-graph can be determined for the range $l\in\{0,\ldots,2k-1\}$ by a polynomial time algorithm known as a \emph{pebble game}~\cite{lee-str}. As such, the $(k,l)$-tight conditions obtained above can be verified in $O(|V|^2)$ time. 
\end{remark}
 
 In the case of admissible norms on $\H_2(\bR)$, the following additional graph properties are necessary for minimal infinitesimal rigidity. Here we regard the simple graph $G$ as a member of the wider class of multi-graphs with no loops. We also recall that a subgraph~$H$ of~$G$ is said to be \emph{spanning} if every vertex of~$G$ is the endpoint of some edge of~$H$.
 
 \begin{corollary}
 Let~$\|\cdot\|$ be an admissible norm on $\H_2(\bR)$ and let $(G,p)$ be a completely full, well-positioned and minimally infinitesimally rigid bar-joint framework in $(\H_2(\bR),\unorm{\cdot})$.
  \begin{enumerate}[(i)]
  \item $G$ can be constructed from  a single vertex using a sequence of graph moves of the following form:

 \begin{itemize}
\item Adjoin a new vertex $v$ which is incident with at most three new edges, at most two of which are parallel.
\item Remove a set $E'$ of $i$ edges, where $i\in\{1,2\}$, and let $V'$ be the set of vertices for edges in $E'$.
 Adjoin a new vertex $v$ which is incident with each vertex in $V'$.
 Adjoin $3-i$ additional edges which are each incident with $v$ and  a vertex not in $V'$, such that no three edges in the resulting multi-graph are parallel.
\end{itemize}
 \item If a single edge is added  to $G$ then the resulting multi-graph is an edge disjoint union of three spanning trees. 
\end{enumerate}
 \end{corollary}
 
 \begin{proof}
By Theorem \ref{thm:tight}, $G$ is $(3,4)$-tight and so $(i)$ is an application of~\cite[Theorem 1.9]{fra-sze} whereas $(ii)$ follows by an argument of Nash-Williams~\cite{n-w} applied to $(3,3)$-tight graphs.
 \end{proof}

For $G=K_m$, it follows from the Maxwell counting criteria (Theorem~\ref{thm:maxwell}) and Lemma~\ref{lem:comb} that a full and well-positioned bar-joint framework $(K_m,p)$ in an admissible matrix space $(X,\|\cdot\|)$ is not infinitesimally rigid  in the following cases:
\begin{enumerate}[(i)]
\item
$X=\M_n(\bR)$ or $\H_n(\bR)$, $k=\dim X$ and $m\in\{2,\ldots, 2k-1\}$.

\item
$X=\M_n(\bC)$ or $\H_n(\bC)$, $k=\dim X$ and $m\in\{2,\ldots, 2k-2\}$.
\end{enumerate}
We make the following conjectures for larger values of $m$. 

\begin{conjecture}
Let $\|\cdot\|$ be an admissible norm on $X\in \{\M_n(\bF),\H_n(\bF)\}$ and let $k=\dim X$.
\begin{enumerate}[(i)]
\item
If~$\bF=\bR$, then there exists $p\in X^V$ such that $(K_m,p)$ is full, well-positioned and infinitesimally rigid in $(X,\|\cdot\|)$ for all $m\geq 2k$.
\item
If~$\bF=\bC$, then there exists $p\in X^V$ such that $(K_m,p)$ is full, well-positioned and infinitesimally rigid in $(X,\|\cdot\|)$ for all $m\geq 2k-1$.
\end{enumerate}
\end{conjecture}

In Section~\ref{sect:applications} we will show that these conjectures 
hold when~$X=\H_2(\bF)$ and the admissible norm is the trace norm.
Namely, we show that there exists $p\in \H_2(\bR)^V$ such that $(K_m,p)$ is full, well-positioned and infinitesimally rigid in $(\H_2(\bR),\|\cdot\|_{c_1})$ for all $m\geq 6$, and, that there exists $p\in \H_2(\bC)^V$ such that $(K_m,p)$ is full, well-positioned and infinitesimally rigid in $(\H_2(\bC),\|\cdot\|_{c_1})$ for all $m\geq 7$.

\section{Product norms}
\label{ProductNorms}
In this section we extend to the setting of product norms a framework colouring technique which was introduced in \cite{kit-pow} to characterise rigidity in $(\bR^d,\|\cdot\|_\infty)$. Our main result is Theorem~\ref{ProjThm}, in which we characterise  infinitesimal rigidity with respect to a product norm in terms of projected monochrome subframeworks. We will apply the results of this section to the admissible matrix space $(\H_2(\bF),\|\cdot\|_{c_1})$ in Section~\ref{sect:applications}.

Let $(X_1,\|\cdot\|_{1}),\ldots,(X_n,\|\cdot\|_{n})$ be a finite collection of finite dimensional real normed linear spaces and let $X=X_1\times \cdots \times X_n$ be the product space. The  \emph{product norm} $\|\cdot\|_\pi$ on $X$ is defined by
\[\|x\|_\pi= \max_{j=1,2,\ldots,n} \, \|x_j\|_{j},\]
for all $x=(x_1,\ldots,x_n)\in X$. For each $j=1,\ldots,n$, denote by $P_j$ the  projection onto $X_j$ given by
\[P_j:X\to X_j, \,\,\,\,\, (x_1,\ldots,x_n)\mapsto x_j,\] and denote
by $P^*_j$ the embedding of $X_j$ into $X$ given by
\[P^*_j:X_j\to X, \,\,\,\,\, y\mapsto (0,\ldots,\overset{j^{\text{th}}}y,\ldots,0).\]
Clearly, $P^*_j$ is an isometry and $\|P_j\|:=\sup\{\|P_j(x)\|_j\colon x\in X,\|x\|_\pi\leq 1\}=1$. Moreover, $\sum_{j=1}^n P_j^* P_j$ and $P_i P_i^*$ are the identity maps on~$X$ and $X_i$, respectively (where, as usual, we write $AB$ for the composition of two linear maps~$A$ and~$B$), and $P_j P_i^*=0$ if $i\ne j$.

\begin{lemma}
\label{Product_SuppFunctional}
Let $x\in X$ with $\|x\|_\pi = \|P_j(x)\|_j=1$.
If $\varphi_j$ is a support functional for $P_j(x)$ in $(X_j,\|\cdot\|_j)$,
then $\varphi=\varphi_j\circ P_j$ is a support functional for $x$ in $(X,\|\cdot\|_\pi)$. 
\end{lemma}

\begin{proof}
We have $\|\varphi\|\leq \|\varphi_j\|\,\|P_j\|=\|\varphi_j\|\le 1$ and $\varphi(x)=1$.
\end{proof}

\subsection{Framework colours}
\label{sect:FrameworkColours}
For $x=(x_1,\ldots,x_n)\in X$, we write
\[\kappa(x) = \big\{j\in\{1,2,\dots,n\}:\|x\|_\pi = \|x_j\|_j\big\}.\]
We think of the non-empty set $\kappa(x)$ as a set of colours assigned to~$x$ by the product norm~$\|\cdot\|_\pi$ .
\begin{lemma}
  \label{lem:Product_limits}
  If~$x$ is a unit vector in~$X$ and $\kappa(x)=\{j\}$ is a singleton,
  then there exists~$\delta>0$ so that $\kappa(x+y)=\{j\}$ whenever
  $y\in X$ with $\|y\|_\pi\leq \delta$.
\end{lemma}
\begin{proof}
  Set $\delta=\tfrac13(\|x_j\|_j-\max\limits_{\{k\colon k\ne
    j\}}\|x_k\|_k)$. Since $\kappa(x)=\{j\}$, we have $\delta>0$. For
  $\|y\|_\pi\le\delta$ and~$k\in \{1,2,\dots,n\}\setminus\{j\}$,
  \[\|(x+y)_j\|_j-\|(x+y)_k\|_k\geq \|x_j\|_j-\|y_j\|_j-(\|x_k\|_k+\|y_k\|_k)\geq 3\delta - 2\|y\|_\pi>0,\]
  so $\kappa(x+y)=\{j\}$.
\end{proof}

Let $(G,p)$ be a bar-joint framework in $(X,\|\cdot\|_\pi)$. There is a natural edge-labelling $\kappa_p$
where for each edge $vw\in E$, we define $\kappa_p(vw) = \kappa(p_v-p_w)$.
An edge $vw\in E$ is said to have \emph{framework colour $j$} if $j\in \kappa_p(vw)$. The set of all edges in $G$ which have framework colour $j$ is denoted $E_j$, and we have $E=E_1\cup\dots\cup E_n$.

Let $G_j=(V,E_j)$ denote the subgraph of $G$ with the same vertex set as $G$ and edge set $E_j$ consisting of all edges with framework colour $j$. We refer to $G_j$ as a \emph{monochrome subgraph} of $G$. Note that $G_j$ may contain vertices of degree $0$ (even if $G$ does not).  
The pair $(G_j,p)$ is a bar-joint framework in $X$ and is referred to as the \emph{monochrome subframework of $(G,p)$ with framework colour $j$}.

For each $j=1,\ldots,n$, we write $p_j=P_j\circ p$.
If $vw\in E_j$, then \[\|p_j(v)-p_j(w)\|_j=\|p(v)-p(w)\|_\pi\ne 0,\] so
$(G_j,p_j)$ is a bar-joint framework in~$X_j$. We call $(G_j,p_j)$ the
\emph{projected monochrome subframework  with framework colour
  $j$}. 

If~$(G_j,p_j)$ is well-positioned
in~$(X_j,\|\cdot\|_j)$ and $vw\in E_j$, then we write
$\varphi_{v,w}^j$ for the support functional at the unit vector
$P_j(p_0)$ in $X_j$, where $p_0=\frac{p_v-p_w}{\|p_v-p_w\|_\pi}$.

\begin{proposition}
\label{prop:ProductWP}
A framework $(G,p)$ in~$(X,\|\cdot\|_\pi)$ is well-positioned in $(X,\|\cdot\|_\pi)$  if and only if 
\begin{enumerate}[(i)]
\item each edge $vw\in E$ has exactly one framework colour, and
\item $(G_j,p_j)$ is well-positioned in $(X_j,\|\cdot\|_j)$  for each $j=1,2,\ldots,n$.
\end{enumerate}
Moreover, in this case we have $\varphi_{v,w}=\varphi_{v,w}^j\circ
P_j$ for every edge~$vw\in E_j$.
\end{proposition}

\begin{proof}
Let $vw\in E$ and write $p_0=\frac{p_v-p_w}{\|p_v-p_w\|_\pi}$. 
If $(G,p)$ is well-positioned in $(X,\|\cdot\|_\pi)$ then by %
Lemma~\ref{lem:flexcondition},
the product norm is smooth at $p_0$ and so $p_0$ has  exactly one support functional.
Suppose $i$ and $j$ are two distinct framework colours for $vw$.
Then $\|p_v-p_w\|_\pi = \|P_i(p_v-p_w)\|_i = \|P_j(p_v-p_w)\|_j$
and so $1=\|p_0\|_\pi = \|P_i(p_0)\|_i = \|P_j(p_0)\|_j$.
Choose support functionals $\varphi_i$ and $\varphi_j$ for $P_i(p_0)$ and $P_j(p_0)$ in $(X_i,\|\cdot\|_i)$ and $(X_j,\|\cdot\|_j)$ respectively.
By Lemma~\ref{Product_SuppFunctional}, 
both $\varphi_i\circ P_i$ and $\varphi_j\circ P_j$ are support functionals for $p_0$, so by smoothness, $\varphi_i\circ P_i=\varphi_j\circ P_j$.
Now $\varphi_i = \varphi_i\circ P_i\circ P^*_i
= \varphi_j\circ P_j\circ P^*_i=0$.
This is a contradiction since $\varphi_i(P_i(p_0))=1$ and so~$(i)$ holds.

Suppose $vw\in E_j$. 
By Lemma~\ref{Product_SuppFunctional}, 
if $\psi_1$ and $\psi_2$ are two support functionals
for $P_j(p_0)$, then $\psi_1\circ P_j$ and $\psi_2\circ P_j$ are both support functionals for $p_0$, hence are equal.
Now $\psi_1 = \psi_1\circ P_j\circ P^*_j=\psi_2\circ P_j\circ P^*_j=\psi_2$ and so $P_j(p_0)$ has exactly one support functional.
Thus the norm $\|\cdot\|_j$ is smooth at $P_j(p_0)$ and 
so $(G_j,p_j)$ is well-positioned in $(X_j,\|\cdot\|_j)$ by Lemma~\ref{lem:flexcondition}.
This proves~$(ii)$.

For the converse, if $(i)$ and $(ii)$ hold then consider an edge
$vw\in E$ and again write $p_0=\frac{p_v-p_w}{\|p_v-p_w\|_\pi}$.  By
$(i)$, $vw$ has a unique framework colour, say~$j$.  By
Lemma~\ref{lem:Product_limits}, for any $z\in X$ we have
\[\lim_{t\to0}\frac{1}{t}(\|p_0+tz\|_\pi-\|p_0\|_\pi)
=\lim_{t\to0}\frac{1}{t}(\|P_j(p_0)+tP_j(z)\|_j-\|P_j(p_0)\|_j),\]
where, by $(ii)$ and Lemma \ref{lem:flexcondition}, the latter limit exists (and is in fact equal to $\varphi_{v,w}^j(P_j(p_0))$). Thus the product norm is smooth at $p_0$ and so $(G,p)$ is well-positioned in $(X,\|\cdot\|_\pi)$.

The final claim follows directly from Lemma~\ref{Product_SuppFunctional}
and the uniqueness of the support functional~$\varphi_{v,w}$.
\end{proof}

If $z:V\to X$, then we write 
$z_j:V\to X_j$, $v\mapsto P_j(z(v))$, and define the linear isomorphism
\[\Phi_V: X^V\to \bigoplus_{j=1}^n X_j^V, \,\,\,\,\,\,\,\,\, z\mapsto (z_1,\ldots,z_n).\]
By Proposition~\ref{prop:ProductWP}, the monochrome edge sets $E_1,\dots,E_n$ arising from a well-positioned bar-joint framework $(G,p)$ partition $E$. Hence, writing $\lambda_j\colon E_j\to \bR$ for the restriction to~$E_j$ of a map $\lambda:E\to \bR$, we have a linear isomorphism
\[\Phi_E: \bR^E\to \bigoplus_{j=1}^n \bR^{E_j}, \,\,\,\,\,\,\,\, \lambda\mapsto (\lambda_1,\ldots, \lambda_n).\]

\begin{corollary}
\label{cor:diff_product}
If $(G,p)$ is well-positioned in $(X,\|\cdot\|_\pi)$, then 
\[df_G(p) = \Phi_E^{-1}\circ (df_{G_1}(p_1)\oplus\cdots \oplus df_{G_n}(p_n))\circ \Phi_V.\]
\end{corollary}

\begin{proof}
Apply Lemma  \ref{lem:flexcondition} and Proposition  \ref{prop:ProductWP}.
\end{proof}

\begin{corollary}
\label{cor:ProductFlex}
Let $(G,p)$ be a well-positioned framework in $(X,\|\cdot\|_\pi)$.
\begin{enumerate}[(i)]
\item $z\in \F(G,p)$ if and only if  $z_j\in\F(G_j,p_j)$ for each $j=1,\ldots,n$.
\item The map
\[\Phi_{(G,p)}:\F(G,p) \to \bigoplus_{j=1}^n \F(G_j,p_j),\,\,\,\,\,\,
z\mapsto (z_1,\ldots,z_n),\]
is a linear isomorphism.
\item $\dim\F(G,p)=\sum_{j=1}^n\dim\F(G_j,p_j)$.
\end{enumerate}
\end{corollary}

\begin{proof}
The statements   follow immediately from Corollary \ref{cor:diff_product} and the observation that $\Phi_{(G,p)}$ is the restriction of $\Phi_V$ to the kernel of $df_G(p)$.
\end{proof}

\subsection{Rigid motions of product spaces}

We will now see that the infinitesimal rigid motions
of~$(X,\|\cdot\|_\pi)$ coincide with direct sums of infinitesimal rigid motions of
the factor spaces~$(X_j,\|\cdot\|_j)$.

Given $\alpha_j\in \R(X_j,\|\cdot\|_j)$ and $\eta_j\in\T(X_j,\|\cdot\|_j)$ for $j=1,\dots,n$, let us define
\[\bigoplus_{j=1}^n \alpha_j\colon X\times [-1,1]\to X,\quad (x,t) \mapsto \sum_{j=1}^n P^*_j \alpha_j(P_j(x),t)\]
and 
\[ \bigoplus_{j=1}^n\eta_j\colon X\to X,\quad x\mapsto \sum_{j=1}^n P_j^* \eta_j(P_j(x)).\]
\begin{lemma}
\label{lem:productrigidmotion} 
For $j=1,\dots,n$, let $\alpha_j\in \R(X_j,\|\cdot\|_j)$ and
let~$\eta_j\in \T(X_j,\|\cdot\|_j)$ be the infinitesimal rigid motion
induced by~$\alpha_j$, and consider
$\alpha=\bigoplus_{j=1}^n \alpha_j$. We have
\begin{enumerate}[(i)]
\item $\alpha\in \R(X,\|\cdot\|_\pi)$; and
\item the infinitesimal rigid motion induced by~$\alpha$ is $\eta:=\bigoplus_{j=1}^n \eta_j$; and
\item $\eta_j=P_j\circ \eta\circ P^*_j$ for each $j=1,\ldots,n$.
\end{enumerate}
\end{lemma}

\begin{proof}
It is clear that $\alpha_x\colon[-1,1]\to X$ is continuous for each $x\in X$
and 
\[\alpha_x(0) = \sum_{j=1}^n P^*_j \alpha_j(P_j(x),0)
=\sum_{j=1}^n P^*_j P_j(x) = x.\]
Also, for any $x,y\in X$ and $t\in[-1,1]$,
\begin{eqnarray*}
\|\alpha_x(t)-\alpha_y(t)\|_\pi
&=& \|\sum_{j=1}^n P^*_j\alpha_j(P_j(x),t)-\sum_{j=1}^n P^*_j\alpha_j(P_j(y),t)\|_\pi \\
&=& \max_{j=1,\ldots,n} \|\alpha_j(P_j(x),t)-\alpha_j(P_j(y),t)\|_j \\
&=& \max_{j=1,\ldots,n} \|P_j(x)-P_j(y)\|_j \\
&=& \|x-y\|_\pi.
\end{eqnarray*}
Note that for each $x\in X$,  
\[\alpha'_x(0)=\sum_{j=1}^n P^*_j ((\alpha_j)_{P_j(x)})'(0) =\sum_{j=1}^n P^*_j\eta_j(P_j(x)).\] 
Thus $\alpha\in \R(X,\|\cdot\|_\pi)$ and $\eta=\sum_{j=1}^n P_j^*\circ \eta_j\circ P_j$ is its induced infinitesimal rigid motion.
Finally, for each $j=1,\ldots,n$,
\[P_j\circ \eta\circ  P_j^*
=\sum_{k=1}^n (P_j P^*_k)\circ \eta_k\circ (P_k P_j^*)
=\eta_j.\qedhere\]
\end{proof}

For a finite dimensional normed vector space~$Y$, we write $\Isom(Y)$
for the set of linear isometries of~$Y$, equipped with the norm topology.

\begin{proposition}\label{prop:isomproduct}
If~$A\colon
  [-1,1]\to \Isom(X)$ is continuous and $A(0)=I_X$, then there exists
  $\delta>0$ so that 
  $A(t)\in \bigoplus_{i=1}^n \Isom(X_i)$ for every~$t\in [-\delta,\delta]$.
\end{proposition}
\begin{proof}
  To simplify notation, we consider the case $n=2$; the general case
  is similar. Write \[A(t)=
  \begin{bmatrix}
    A_{11}(t)&A_{12}(t)\\A_{21}(t)&A_{22}(t)
  \end{bmatrix}\] where each $A_{ij}(t)$ is a linear map from~$X_j$
  to~$X_i$. It is then easy to see that $\|A_{ij}(t)\|\leq \|A(t)\|=1$
  for every~$i,j$ and~$t$. We first claim that $A_{jj}(t)\in
  \Isom(X_j)$ for $j=1,2$ and $|t|$ sufficiently small. If not, taking
  $j=1$ without loss of generality, there exist $t_n\to 0$ and
  $x_n\in X_1$ with $\|x_n\|_1=1$ so that $\|A_{11}(t_n)x_n\|_1<1$ for
  every~$n\in \bN$. Now
  \[ 1=\left\|A(t_n)
    \begin{bmatrix}
      x_n\\0
    \end{bmatrix}\right\|_\pi = \left\|
    \begin{bmatrix}
      A_{11}(t_n)x_n\\A_{21}(t_n)x_n
    \end{bmatrix}\right\|_\pi = \max\{     \|A_{11}(t_n)x_n\|_1,\|A_{21}(t_n)x_n\|_2\}\]
  so we necessarily have $\|A_{21}(t_n)x_n\|_2=1$ for every~$n$. However,
  \[ \|A_{21}(t_n)x_n\|_2 \leq \|A_{21}(t_n)\| \to 0\text{ as~$n\to \infty$},\]
  a contradiction which establishes the claim.
  
  Hence there exists~$\delta>0$ so that $A_{jj}(t)\in \Isom(X_j)$ for
  $j=1,2$ and $|t|\le\delta$. We now claim that $A_{ij}(t)=0$ for
  $|t|\le\delta$ and $i\ne j$. We show this for $(i,j)=(1,2)$, and the
  other case follows by symmetry. Suppose instead that this claim
  fails at some~$t\in [-\delta,\delta]$, and write $A=A(t)$ and
  $A_{ij}=A_{ij}(t)$. Since $A_{12}\ne 0$, we can find a unit
  vector~$y\in X_2$ with $A_{12}y\ne 0$. Since~$A_{11}$ is an
  isometry, it is invertible; let $x=\|A_{12}y\|_1^{-1}A_{11}^{-1}A_{12}y$. Since~$A_{11}$ is an isometry, we have $\|x\|_1=1$, so 
  \begin{align*} 
    \left\|A
      \begin{bmatrix}
        x\\y
      \end{bmatrix}\right\|_\pi&=
    \left\|
      \begin{bmatrix}
        x\\y
      \end{bmatrix}\right\|_\pi = \max\{\|x\|_1,\|y\|_2\}=1.
      \end{align*}
      On the other hand,
  \begin{align*} 
    \left\|A
      \begin{bmatrix}
        x\\y
      \end{bmatrix}\right\|_\pi&=
    \left\|\begin{bmatrix}
        (\|A_{12}y\|_1^{-1}+1)A_{12}y\\{}*
      \end{bmatrix}\right\|_{\pi}\\&\geq \|    (\|A_{12}y\|_1^{-1}+1)A_{12}y\|_1 
    \\&= 1+\|A_{12}y\|_1 >1\end{align*} where $*$ denotes an
  unimportant matrix entry. This contradiction shows that
  $A_{12}=0$. Hence $A(t)=A_{11}(t)\oplus A_{22}(t)$ for $|t|\leq\delta$, as desired.
\end{proof}

\begin{theorem}\label{thm:productflex}
  $\T(X,\|\cdot\|_\pi)=\bigoplus_{i=1}^n \T(X_i,\|\cdot\|_i)$.
\end{theorem}
\begin{proof}
  The inclusion ``$\supseteq$'' follows from
  Lemma~\ref{lem:productrigidmotion}. For the reverse inclusion, let
  $\eta\in \T(X,\|\cdot\|_\pi)$ and choose
  $\alpha\in \R(X,\|\cdot\|_\pi)$ which induces~$\eta$.  By the
  Mazur-Ulam theorem, we may write \[\alpha(x,t)=A(t)x+c(t)\] where
  $c\colon [-1,1]\to X$ and $A\colon [-1,1]\to \Isom(X)$ are
  continuous and differentiable at~$t=0$ with $c(0)=0$ and
  $A(0)=I_X$. Choose~$\delta>0$ by Proposition~\ref{prop:isomproduct},
  so that $A(t)=\bigoplus_{i=1}^n A_i(t)$ for~$t\in [-\delta,\delta]$
  where $A_i\colon [-\delta,\delta]\to \Isom(X_i)$ are continuous and
  differentiable at~$t=0$.  Consider the map
   \[\alpha_i\colon X_i\times[-1,1]\to X_i,\quad \alpha_i(y,t)= P_i\alpha(P_i^*y,\tau)\quad\text{where $\tau=\min\{\delta,t\}$}.\]
   This map is plainly continuous in~$t$ and differentiable at~$t=0$,
   and it is isometric in~$y$ since
  \begin{align*}
    \|\alpha_i(y,t)-\alpha_i(z,t)\|_i&=\|A_i(\tau)y+P_ic(\tau)-(A_i(\tau)z+P_ic(\tau))\|_i\\
    &=\|A_i(y-z)\|_i=\|y-z\|_i.
  \end{align*}
  Hence $\alpha_i\in\R(X_i,\|\cdot\|_i)$.  Moreover, for $|t|\leq \delta$ and~$x\in X$,
  we have 
  \begin{align*}
    \sum_{j=1}^n P_j^* \alpha_j(P_j(x),t) &= \sum_{j=1}^n P_j^*P_j \alpha(P_j^*P_j x,t) = \sum_{j=1}^n P_j^* (A_j(t)P_jx+P_jc(t))\\&=A(t)x+c(t)=\alpha_x(t).
  \end{align*}
  This shows that on a neighbourhood of~$t=0$, the rigid motion
  $\alpha$ coincides with~$\bigoplus_{j=1}^n \alpha_j$.  Hence
  $\alpha$ and~$\bigoplus_{j=1}^n \alpha_j$~induce the same
  infinitesimal rigid motion, namely~$\eta$. Hence
  $\eta\in\bigoplus_{j=1}^n
  \T(X_j,\|\cdot\|_j)$ by Lemma~\ref{lem:productrigidmotion}.
\end{proof}

As a corollary we obtain the following characterisation for full sets in product spaces.

\begin{corollary}
\label{cor:full_product}
A set $S\subset X$ is full in $(X,\|\cdot\|_\pi)$ if and only if $P_j(S)$ is full in $(X_j,\|\cdot\|_j)$ for each $j=1,\ldots,n$.
\end{corollary}

\subsection{Trivial infinitesimal flexes of frameworks}
Let $(G,p)$ be a %
bar-joint framework in $(X,\|\cdot\|_\pi)$.  
 For $j=1,\dots,n$ and~$z\in \T(G,p)$ (so that $z\colon V\to X$), we define (as before)
 $z_j=P_j\circ z\colon V\to X_j$.%

\begin{proposition}
\label{prop:ProductTriv}
Let $(G,p)$ be a %
bar-joint framework in $(X,\|\cdot\|_\pi)$.
\begin{enumerate}[(i)]
\item
If $z\in \T(G,p)$, then $ z_j\in\T(G_j,p_j)$ for each $j=1,\ldots,n$.
\item
The map
\[\tilde \Phi_{(G,p)}:\T(G,p)\to\bigoplus_{j=1}^n\T(G_j,p_j),
\,\,\,\,\,
z\mapsto ( z_1,\ldots, z_n)\]
is a linear isomorphism.
\item $\dim \T(G,p)=\sum_{j=1}^n \dim \T(G_j,p_j)$.
\end{enumerate}
\end{proposition}

\begin{proof}
$(i)$
Since~$z\in \T(G,p)$, there exists %
an infinitesimal rigid motion $\eta\in \T(X,\|\cdot\|_\pi)$ with $z(v)=\eta(p_v)$ for each $v\in V$. By Theorem~\ref{thm:productflex}, $\eta=\bigoplus_{i=1}^n\eta_i$ where $\eta_i\in \T(X_i,\|\cdot\|_i)$ for $i=1,\dots,n$, so
\[ z_j(v)=P_j\left(\bigoplus_{i=1}^n\eta_i\right)(p_v)=\eta_j(P_j(p_v))=\eta_j(p_j(v)).\]
Thus $z_j$ is the trivial infinitesimal flex of $(G_j,p_j)$ induced by the infinitesimal rigid motion $\eta_j$, so $z_j\in \T(G_j,p_j)$.

$(ii)$   By $(i)$ this map is well defined, and it is easily seen to be linear. Since $z=\sum_{j=1}^n P_j^*\circ P_j\circ z=\sum_{j=1}^n P_j^*\circ z_j$ for any~$z\in \T(G,p)$, we see immediately that~$\tilde \Phi_{(G,p)}$ is injective.
  For surjectivity, observe that if $w=(w_1,\ldots,w_n)\in \oplus_{j=1}^n\T(G_j,p_j)$, then $w_j=\eta_j\circ p_j$ for some $\eta_j\in \T(X_j,\|\cdot\|_j)$, hence $w_j=\eta_j\circ P_j\circ p$. We have $\eta:=\bigoplus_{j=1}^n \eta_j\in \T(X,\|\cdot\|_\pi)$ by Theorem~\ref{thm:productflex}, so $\eta\circ p\in \T(G,p)$. Let $j\in\{1,\dots,n\}$. We have $P_j\circ \eta=\eta_j\circ P_j$, so the $j$th component of $\tilde \Phi_{(G,p)}(\eta\circ p)$ is 
  $P_j\circ \eta\circ p = \eta_j\circ P_j \circ p=w_j$,
  hence $\tilde\Phi_{(G,p)}(\eta\circ p)=w$. Assertion~$(iii)$ follows immediately.
\end{proof}

\subsection{A characterisation of infinitesimal rigidity}
We can now characterise infinitesimal rigidity for well-positioned bar-joint frameworks in terms of their projected monochrome subframeworks. 

\begin{theorem}
\label{ProjThm}
If $(G,p)$ is a well-positioned %
bar-joint framework in $(X,\|\cdot\|_\pi)$, then
the following statements are equivalent.
\begin{enumerate}[(i)]
\item
$(G,p)$ is (minimally) infinitesimally rigid in $(X,\|\cdot\|_\pi)$. 
\item
The projected monochrome subframeworks $(G_j,p_j)$ are (minimally) infinitesimally rigid in $(X_j,\|\cdot\|_j)$ for each $j=1,2,\ldots,n$.
\end{enumerate}
\end{theorem}

\begin{proof}
The statement follows from Corollary~\ref{cor:ProductFlex} and Proposition~\ref{prop:ProductTriv}. Indeed, if $(G,p)$ is infinitesimally rigid then
\[\sum_{j=1}^n \dim \F(G_j,p_j) = 
\dim \F(G,p)= \dim \T(G,p)=\sum_{j=1}^n \dim \T(G_j,p_j).\]
and, since $\T(G_j,p_j)$ is a subspace of $\F(G_j,p_j)$ for each $j$, condition $(ii)$ follows. Conversely, if $(ii)$ holds then 
\[\dim \F(G,p)=\sum_{j=1}^n \dim \F(G_j,p_j) = \sum_{j=1}^n \dim \T(G_j,p_j)= \dim \T(G,p).\]
and so condition $(i)$ follows.
\end{proof}

The following result was obtained by different methods in \cite{kit-pow}. 

\begin{corollary}
Let $(G,p)$ be a well-positioned  framework in $(\bR^d,\|\cdot\|_\infty)$.
The following statements are equivalent.
\begin{enumerate}[(i)]
\item $(G,p)$ is minimally infinitesimally rigid in $(\bR^d,\|\cdot\|_\infty)$.
\item The monochrome subgraphs $G_1,\ldots,G_d$ are spanning trees in $G$.
\end{enumerate}
\end{corollary}

\begin{proof}
By Theorem~\ref{ProjThm}, $(G,p)$ is infinitesimally rigid if and only if each $(G_j,p_j)$ is infinitesimally rigid. The result now follows from the observation that a framework in $\bR$ is (minimally) infinitesimally rigid if and only if the underlying graph is connected (respectively, a tree).
\end{proof}

\section{Application to \texorpdfstring{$(\H_2(\bF),\|\cdot\|_{c_1})$}{H2(F),||c1}}
\label{sect:applications}
In this section we apply Theorem \ref{ProjThm} to characterise infinitesimal rigidity in  the matrix space $\H_2(\bF)$ endowed with the trace norm, for both $\bF=\bR$ and $\bF=\bC$. These normed spaces can be identified, under a suitable isometric isomorphism, with a product norm on $\bR^3$ or $\bR^4$ respectively. We also show how to construct an infinitesimally rigid placement of the complete graph $K_m$ in $(\H_2(\bF),\|\cdot\|_{c_1})$ for sufficiently large values of  $m$.

\subsection{Symmetric matrices}
\newcommand{\cyl}{\text{\upshape{cyl}}}
Denote by $\|\cdot\|_{\cyl}$ the product norm on $\bR^3=X_1\times X_2$, where $X_1=\bR^2$ and $X_2=\bR$, given by
\[\|(x,y,z)\|_{\cyl} = \max \{ \sqrt{x^2+y^2},|z|\}.\]
Note that the closed unit ball in $(\bR^3,\|\cdot\|_{\cyl})$ is a cylinder $D\times [-1,1]$ where $D$ is the closed unit disk in the Euclidean plane. 
We refer to a normed linear space which is isometrically isomorphic to $(\bR^3,\|\cdot\|_{\cyl})$ as a \emph{cylindrical normed space}.

\begin{lemma}
\label{CylLemma}
\begin{enumerate}[(i)]
\item $(\H_2(\bR),\|\cdot\|_{c_1})$ is a cylindrical normed space. 

\item Every cylindrical normed space $(X,\|\cdot\|)$ 
satisfies $\dim \T(X, \|\cdot\|) =4$.

\item In a cylindrical normed space, every bar-joint framework on a graph containing at least one edge is full.

\end{enumerate}
\end{lemma}

\begin{proof}
The map 
\[\Psi: (\bR^3,\|\cdot\|_{\cyl}) 
\to (\H_2(\bR),\|\cdot\|_{c_1}), \,\,\,\,\,\,
(x,y,z)\mapsto \frac{1}{2}\left(\begin{array}{cc}
z+y & x \\
x & z-y
\end{array}
\right)
\]
is an isometric isomorphism. Indeed, the eigenvalues of $\Psi(x,y,z)$ are $\lambda_{\pm}=\tfrac12(z\pm\sqrt{x^2+y^2})$, hence $\|\Psi(x,y,z)\|_{c_1}=|\lambda_+|+|\lambda_-|=\|(x,y,z)\|_{\cyl}$.
Statement $(ii)$ follows from the corresponding property of $(\H_2(\bR),\|\cdot\|_{c_1})$, established in Proposition~\ref{prop:dim2}.
Statement $(iii)$ follows from Corollary \ref{cor:full_product} and the easily verified fact that once we have at least one edge, every bar-joint framework in the Euclidean plane, and every bar-joint framework in $\bR$, is full.
\end{proof}

\begin{lemma}
\label{lem:cylwp}
Let $(G,p)$ be a bar-joint framework in $(\bR^3,\|\cdot\|_{\cyl})$.

\begin{enumerate}[(i)]
\item For $p_v-p_w=(x,y,z)$ where $vw\in E$, we have $1\in \kappa_p(vw)$ if and only if $x^2+y^2\geq z^2$, and $2\in \kappa_p(vw)$ if and only if $x^2+y^2\leq z^2$.

\item
$(G,p)$ is well-positioned in $(\bR^3,\|\cdot\|_{\cyl})$ if and only if $p_v-p_w$ does not lie in the cone $C=\{(x,y,z)\in\bR^3: x^2+y^2=z^2\}$ for each edge $vw\in E$. 
\end{enumerate}
\end{lemma}

\begin{proof}
Part $(i)$ follows immediately from the definitions in Section~\ref{sect:FrameworkColours}. Since the Euclidean norm is smooth, every bar-joint framework in the Euclidean plane and every bar-joint framework in $\bR$ is well-positioned. Also note that an edge $vw\in E$ has exactly one framework colour if and only if $p_v-p_w\notin C$.
Thus $(ii)$ follows from Proposition~\ref{prop:ProductWP}.
\end{proof}

Let $\omega=\omega(G,X,\|\cdot\|)\subset X^V$ denote the set of all well-positioned placements of a graph $G$ in a normed space $(X,\|\cdot\|)$.
A placement $p\in\omega$ is said to be {\em regular} if the function,
\[\omega\to \{1,\ldots,|E|\}, \,\,\,\,\,\, x\mapsto \rank\,df_G(x),\]
achieves its maximum value at $p$.

\begin{remark}
\label{rem:regular1}
The set $\omega(G,\bR^d,\|\cdot\|_2)$ of regular placements for a graph $G=(V,E)$ in Euclidean space is an open and dense subset of $(\bR^d)^V$. Moreover, if $G$ admits an infinitesimally rigid placement in $(\bR^d,\|\cdot\|_2)$ then all regular placements of $G$ in $(\bR^d,\|\cdot\|_2)$ are infinitesimally rigid. (See \cite[p.283 and Corollary 2]{asi-rot} for example). In this case, $G$ is said to be {\em generically rigid} in $(\bR^d,\|\cdot\|_2)$.
\end{remark}

A graph is said to be a \emph{Laman graph} if it is $(2,3)$-tight.

 \begin{theorem}
 \label{thm:cyl}
 Let $(G,p)$ be a well-positioned  bar-joint framework in the cylindrical normed space $(\bR^3,\|\cdot\|_{\cyl})$.
 The following statements are equivalent.
 \begin{enumerate}[(i)]
 \item $(G,p)$ is minimally infinitesimally rigid in $(\bR^3,\|\cdot\|_{\cyl})$.
 \item The projected monochrome subframeworks $(G_1,p_1)$ and $(G_2,p_2)$ are  minimally infinitesimally rigid in the Euclidean plane and the real line respectively.
 \item The monochrome subgraphs $G_1$ and $G_2$ are respectively a Laman graph and a tree, and $p_1$ is a regular placement of $G_1$ in the Euclidean plane. 
 \end{enumerate}
 \end{theorem}

\begin{proof}
Theorem~\ref{ProjThm} shows that $(i)$ and $(ii)$ are equivalent.
The equivalence of $(ii)$ and $(iii)$ is an application of standard results on infinitesimal rigidity for bar-joint frameworks in Euclidean space.
See for example \cite[\S3-4]{asi-rot}, Laman~\cite[Theorem 6.5]{laman} and \cite[Propositions 2.4 and 2.5]{whi84}. 
\end{proof}

The following theorem shows that $K_6-e$, the complete graph $K_6$ with a single edge removed, is the smallest graph which admits a well-positioned and minimally infinitesimally rigid bar-joint framework in a cylindrical normed space.

\begin{theorem}
\label{thm:cylmain}
Let $(X,\|\cdot\|)$ be a cylindrical normed space.
\begin{enumerate}[(i)]
\item If~$(G,p)$ is a  well-positioned and minimally infinitesimally rigid bar-joint framework in $(X,\|\cdot\|)$, then either $G=K_6-e$ or $|V(G)|\geq 7$.

\item There is a placement $p$ of $K_6-e$ in~$X$ so that $(K_6-e,p)$ is  well-positioned and minimally infinitesimally rigid in~$(X,\|\cdot\|)$. 
\end{enumerate}
\end{theorem}

\begin{proof}
$(i)$ By Lemma \ref{CylLemma}, $(G,p)$ is full and we may assume, without loss of generality, that $(X,\|\cdot\|)=(\H_2(\bR),\|\cdot\|_{c_1})$. By the Maxwell condition for $(\H_2(\bR),\|\cdot\|_{c_1})$ given in Theorem \ref{thm:maxwell}$(ii)$, we have $|E|=3|V|-4$. 
If $|V|\in\{2,3,4,5\}$ then, by Lemma \ref{lem:comb}, 
$|E|<3|V|-4$. Thus $|V|\geq 6$.
If $|V|=6$ then $|E|=3|V|-4=14=\binom 62-1$ edges. Thus $G=K_6-e$.

$(ii)$ 
It is sufficient to construct such a placement of $K_6-e$ in  
$(\bR^3,\|\cdot\|_{\cyl})$.
Let $V=\{v_i\colon 1\leq i\leq 6\}$ be the vertex set of~$G=K_6-e$ where $e=v_5v_6$.
Let $\epsilon,\delta\in(0,\tfrac12)$  and consider the placement $p:V\to\bR^3$ with
\begin{align*} 
  v_1&\mapsto(0,-1,-1),&
  v_2&\mapsto(0,1,-1)\\
  v_3&\mapsto(0,1,1+2\epsilon),&
  v_4&\mapsto(0,-1,1-2\epsilon)\\
v_5&\mapsto (2\delta,1,-1),& v_6&\mapsto (2\delta ,-1,1-2\epsilon).
\end{align*}
A calculation using Lemma~\ref{lem:cylwp} shows that $(G,p)$ is well-positioned, with monochrome subgraphs~$G_1=\kappa_p^{-1}(\{1\})$ and~$G_2=\kappa_p^{-1}(\{2\})$ as indicated in Figures~\ref{fig:K6-e-colouring} and~\ref{fig:mono-subgraphs}. Note that $G_1$ is a Laman graph and $G_2$ is a tree.
We claim that $p_1$ is a regular placement of $G_1$ in the Euclidean plane. This is an exercise in elementary planar rigidity. The rank of the differential $df_{G_1}(p_1)$ may be computed as the rank of an associated (Euclidean) {\em rigidity matrix} $R(G_1,p_1)$ with rows indexed by 
$E(G_1)$
and block columns indexed by~$V$. The $(v_iv_j,v_i)$-entry, for each edge $v_iv_j$, is the row vector $(p_1(v_i)-p_1(v_j))^t$. All remaining entries are zero. (See \cite[Chapter~2]{gra-ser-ser}). In this case, ordering the edges of~$G_1$ as $(15,  45, 25, 12, 46, 26, 24, 34, 36)$ we obtain
\[\newcommand{\hl}[1]{#1}%
R(G_1,p_1)=
2\left[\begin{smallmatrix}
    \hl{-\delta} & -1      & 0            & 0           & 0       & 0      & 0            & 0       & \delta      & 1           & 0           & 0     \\
    0            & 0       & 0            & 0           & 0       & 0      & -\delta      & -1 & \delta      & \hl{1} & 0           & 0     \\
    0            & 0       & -\delta      & 0           & 0       & 0      & 0            & 0       & \hl{\delta} & 0           & 0           & 0     \\
    0            & \hl{-1} & 0            & 1           & 0       & 0      & 0            & 0       & 0           & 0           & 0           & 0     \\
    0            & 0       & 0            & 0           & 0       & 0      & \hl{-\delta} & 0    & 0           & 0           & \delta      & 0 \\
    0            & 0       & \hl{-\delta} & 1           & 0       & 0      & 0            & 0       & 0           & 0           & \delta      & -1    \\
    0            & 0       & 0            & \hl{1} & 0       & 0      & 0            & -1 & 0           & 0           & 0           & 0     \\
    0            & 0       & 0            & 0           & 0       & 1      & 0            & \hl{-1} & 0           & 0           & 0           & 0     \\
    0            & 0       & 0            & 0           & -\delta & 1 & 0            & 0       & 0           & 0           & \hl{\delta} & -1 
\end{smallmatrix}\right]
\]
Note that each row contains a nonzero entry with only zeros below. Hence, the rows of $R(G_1,p_1)$ are linearly independent and the rank of the differential $df_{G_1}(x)$ at $p_1$ is maximal. Thus $p_1$ is a regular placement of $G_1$ and so,  by Theorem~\ref{thm:cyl}, $(G,p)$ is minimally infinitesimally rigid in $(\bR^3,\|\cdot\|_{\cyl})$.
\end{proof}

  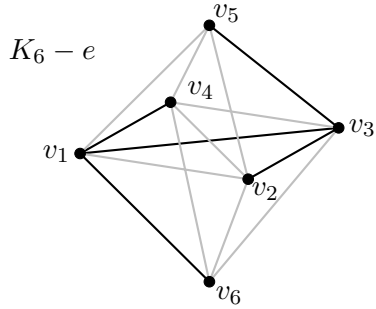
\begin{figure}[ht]
 \centering
 \begin{tabular}{  l c r  }

\hspace{10mm}
\begin{minipage}{.6\textwidth}
  \begin{tikzpicture}[scale=1.7]
 
  \clip (-1.8,-1.3) rectangle (1.5cm,1.3cm); %
  
  \coordinate (A1) at (-1,0);
  \coordinate (A2) at (0.3,-0.2);
	\coordinate (A3) at (1,0.2);
  \coordinate (A4) at (-0.3,0.4);
  \coordinate (B) at (0,1);
  \coordinate (C) at (0,-1);
  
 \draw[thick] (B) -- (A3) -- (A1) -- (C);
 \draw[thick] (A3) -- (A2);
 \draw[thick] (A1) -- (A4);
 \draw[thick,lightgray] (B) -- (A1) -- (A2) -- (C) -- (A4) -- (B) -- (A2);
 \draw[thick,lightgray] (C) -- (A3) -- (A4) -- (A2);
	
  \node[draw,circle,inner sep=1.4pt,fill] at (A1) {};
  \node[draw,circle,inner sep=1.4pt,fill] at (A3) {};
  \node[draw,circle,inner sep=1.4pt,fill] at (A4) {};
  \node[draw,circle,inner sep=1.4pt,fill] at (A2) {};
  \node[draw,circle,inner sep=1.4pt,fill] at (B) {};
  \node[draw,circle,inner sep=1.4pt,fill] at (C) {};

  \node[left] at (-0.8,0.8) {$K_6-e$};
	
  \node[left] at (A1) {$v_1$};
  \node[right] at (A3) {$v_3$};
	\node[above right] at (-0.25,0.35) {$v_4$};
  \node[below right] at (0.25,-0.15) {$v_2$};
	\node[above right] at (-0.05,0.95) {$v_5$};
  \node[below right] at (-0.03,-0.96) {$v_6$};
  \end{tikzpicture}
\end{minipage}

 \end{tabular}
\caption{The framework colouring of $K_6-e$ in the proof of Theorem~\ref{thm:cylmain}. 
The monochrome subgraphs $G_1$ and $G_2$ are indicated in gray and black respectively.}
\label{fig:K6-e-colouring}
\end{figure}

\begin{figure}[ht]
 \centering
 \begin{tabular}{  l c r  }

 \begin{minipage}{.35\textwidth}
 \begin{tikzpicture}[scale=1.2]
 
    \clip (-2,-0.7) rectangle (1.8cm,1.5cm); %
  
  \coordinate (A5) at (-1,1);
  \coordinate (A6) at (1,0);
  \coordinate (A1) at (-1,0);
  \coordinate (A2) at (0,0);
  \coordinate (A3) at (1,1);
  \coordinate (A4) at (0,1);
  
 \draw[thick,lightgray] (A5) -- (A1) -- (A2) -- (A5) -- (A4) -- (A2) -- (A6) -- (A4) -- (A3) -- (A6);

  \node[draw,circle,inner sep=1.4pt,fill] at (A5) {};
  \node[draw,circle,inner sep=1.4pt,fill] at (A6) {};
  \node[draw,circle,inner sep=1.4pt,fill] at (A1) {};
  \node[draw,circle,inner sep=1.4pt,fill] at (A2) {};
  \node[draw,circle,inner sep=1.4pt,fill] at (A3) {};
  \node[draw,circle,inner sep=1.4pt,fill] at (A4) {};

  \node[left] at (-1.4,0.7) {$G_1$};
  
	\node[left] at (A5) {$v_5$};
  \node[right] at (A6) {$v_6$};
	\node[left] at (A1) {$v_1$};
  \node[below] at (A2) {$v_2$};
	\node[right] at (A3) {$v_3$};
  \node[above] at (A4) {$v_4$};
  \end{tikzpicture}
\end{minipage}

\hspace{10mm}
 \begin{minipage}{.35\textwidth}
 \begin{tikzpicture}[scale=1.2]
 
    \clip (-2,-0.7) rectangle (1.8cm,1.5cm); %
  
  \coordinate (A5) at (1,1);
  \coordinate (A6) at (-1,0);
  \coordinate (A1) at (-1,0.5);
  \coordinate (A2) at (1,0);
  \coordinate (A3) at (1,0.5);
  \coordinate (A4) at (-1,1);
  
\draw[thick] (A4) -- (A1) -- (A3) -- (A5);
\draw[thick] (A6) -- (A1);
\draw[thick] (A2) -- (A3);

  \node[draw,circle,inner sep=1.4pt,fill] at (A1) {};
  \node[draw,circle,inner sep=1.4pt,fill] at (A2) {};
  \node[draw,circle,inner sep=1.4pt,fill] at (A3) {};
  \node[draw,circle,inner sep=1.4pt,fill] at (A4) {};
  \node[draw,circle,inner sep=1.4pt,fill] at (A5) {};
  \node[draw,circle,inner sep=1.4pt,fill] at (A6) {};

  \node[left] at (-1.45,0.7) {$G_2$};
 
	\node[right] at (A5) {$v_5$};
  \node[left] at (A6) {$v_6$};
	\node[left] at (A1) {$v_1$};
  \node[right] at (A2) {$v_2$};
	\node[right] at (A3) {$v_3$};
  \node[left] at (A4) {$v_4$};
	
  \end{tikzpicture}
\end{minipage}

\end{tabular}
\caption{The monochrome subgraphs $G_1$ and $G_2$ of $K_6-e$ constructed in Theorem~\ref{thm:cylmain} are respectively a Laman graph and a tree.}
\label{fig:mono-subgraphs}
\end{figure}
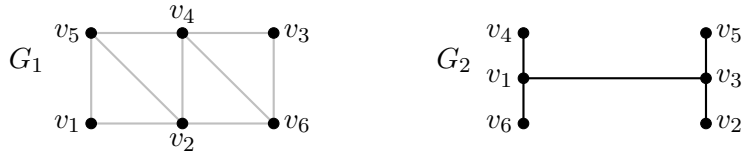

\begin{remark}
Applying the isometric  isomorphism~$\Psi$ from the proof of Lemma~\ref{CylLemma} to the framework constructed in Theorem~\ref{thm:cylmain}, we obtain the following matrices which, for $\epsilon,\delta\in (0,\tfrac12)$  %
form a minimally infinitesimally rigid framework $(K_6-e,p)$ in $(\H_2(\bR),\|\cdot\|_{c_1})$, where $e=v_5v_6$.
\begin{align*}
p_{v_1}&=\left(\begin{array}{cc}
-1         & 0      \\
0          & 0
\end{array}\right),&
p_{v_2}&=\left(\begin{array}{cc}
0          & 0      \\
0          & -1
\end{array}\right),
\\
p_{v_3}&=\left(\begin{array}{cc}
1+\epsilon & 0      \\
0          & \epsilon
\end{array}\right),&
p_{v_4}&=\left(\begin{array}{cc}
-\epsilon  & 0      \\
0          & 1-\epsilon
\end{array}\right),
\\
p_{v_5}&=\left(\begin{array}{cc}
0         & \delta \\
\delta     & -1
\end{array}\right),&
p_{v_6}&=\left(\begin{array}{cc}
-\epsilon & \delta \\
\delta     & 1-\epsilon
\end{array}\right).
\end{align*}
\end{remark}

\begin{theorem}
\label{thm:Km1}
Let $(X,\|\cdot\|)$ be a cylindrical normed space.
If~$m\ge6$, then there is a placement $p$ of the complete graph $K_m$ in $X$ so that $(K_m,p)$ is  well-positioned and infinitesimally rigid in $(X,\|\cdot\|)$.
\end{theorem}

\begin{proof}
  Again, it is sufficient to construct such a placement of $K_m$ in  
$(\bR^3,\|\cdot\|_{\cyl})$.
Consider the  well-positioned and minimally infinitesimally rigid framework $(K_6-e,p)$ constructed in Theorem~\ref{thm:cylmain}, with corresponding induced monochrome subgraphs $G_1$ and $G_2$ of $G=K_6-e$. Since $p(v_5)\ne p(v_6)$, the placement $p$ also yields a  bar-joint framework $(K_6,p)$, with respect to which $\kappa_p(v_5v_6)=\{1\}$. Thus $(K_6,p)$ is  well-positioned and infinitesimally rigid in $(\bR^3,\|\cdot\|_{\cyl})$. 

Now consider the complete graph $K_7$ obtained by adjoining a vertex $v_7$ to $K_6$. We will show that we can extend $p$ to a suitable placement of $K_7$ by choosing $p(v_7)$ to be a small perturbation of $p(v_5)$.
By Lemma~\ref{lem:cylwp}, there is an open neighbourhood $U$ of $p(v_5)$ which does not contain $p(v_i)$ for $1\leq i\leq 6$ with $i\ne 5$, such that for any choice of $p(v_7)$ in~$U\setminus\{p(v_5)\}$, the extended bar-joint framework $(K_7,p)$ 
is well-positioned and satisfies $\kappa_p(v_iv_7)=\kappa_p(v_iv_5)$  for $i=1,2,3,4,6$. 
Let $G_1'$ and $G_2'$ be the induced monochrome subgraphs of $K_7$ with framework colours $1$ and $2$ respectively.
Note that $G_2'$  contains a spanning tree obtained by adjoining the vertex $v_7$ and the edge $v_3v_7$ to $G_2$. Also observe that $G_1'$ has a spanning subgraph, obtained by adjoining the vertex $v_7$ and the edges $v_1v_7$, $v_2v_7$ to $G_1$, which is also a Laman graph, hence is minimally infinitesimally rigid in~$(\bR^2,\|\cdot\|_2)$. By Remark~\ref{rem:regular1}, every regular placement of $G_1'$ in $(\bR^2,\|\cdot\|_2)$ is infinitesimally rigid and the set of regular placements for $G_1'$ is dense in $(\bR^2)^V$, so we may choose $p(v_7)$ in~$U\setminus\{p(v_5)\}$ such that $p_1$ is a regular placement of $G_1'$. 
Thus, by %
Theorem~\ref{thm:cyl},
$(K_7,p)$ has a minimally infinitesimally rigid subframework and so is itself infinitesimally rigid. 

We can now apply this method iteratively to obtain a well-positioned and infinitesimally rigid placement of $K_m$ for any $m>6$. 
\end{proof}

\subsection{Hermitian matrices}
\newcommand{\hcyl}{\text{\upshape{hcyl}}}
Similar methods may be applied in the case of $(\H_2(\bC),\|\cdot\|_{c_1})$. Denote by $\|\cdot\|_{\hcyl}$ the product norm on $\bR^4=\bR^3\times \bR$ given by
\[\|(w,x,y,z)\|_{\hcyl} = \max \{ \sqrt{w^2+x^2+y^2},|z|\}.\]
A normed space which is isometrically isomorphic to $(\bR^4,\|\cdot\|_{\hcyl})$ will be referred to as a \emph{hyper-cylindrical normed space}.

\begin{lemma}
\label{HCylLemma}
\begin{enumerate}[(i)]
\item $(\H_2(\bC),\|\cdot\|_{c_1})$ is a hyper-cylindrical normed space. 

\item Every hyper-cylindrical normed space $(X,\|\cdot\|)$ 
satisfies $\dim \T(X, \|\cdot\|) =7$.

\end{enumerate}
\end{lemma}

\begin{proof}
The map
\[\Psi: (\bR^4,\|\cdot\|_{\hcyl}) 
\to (\H_2(\bC),\|\cdot\|_{c_1}), \,\,\,\,\,\,
(w,x,y,z)\mapsto \frac{1}{2}\left(\begin{array}{cc}
z+y & x-wi \\
x+wi & z-y
\end{array}
\right)
\]
is an isometric isomorphism. 
Statement $(ii)$ follows from the corresponding property of $(\H_2(\bC),\|\cdot\|_{c_1})$, established in Proposition~\ref{prop:dim2}.
\end{proof}

\begin{lemma}
\label{lem:hcylwp}
Let $(G,p)$ be a bar-joint framework in $(\bR^4,\|\cdot\|_{\hcyl})$.

\begin{enumerate}[(i)]
\item For $p_v-p_w=(u,x,y,z)$ where $vw\in E$, we have $1\in \kappa_p(vw)$ if and only if $u^2+x^2+y^2\geq z^2$, and $2\in \kappa_p(vw)$ if and only if $u^2+x^2+y^2\leq z^2$.

\item
$(G,p)$ is well-positioned in $(\bR^4,\|\cdot\|_{\hcyl})$ if and only if $p_v-p_w$ does not lie in the cone $C=\{(u,x,y,z)\in\bR^4: u^2+x^2+y^2=z^2\}$ for each edge $vw\in E$. 
\end{enumerate}
\end{lemma}

\begin{proof}
The proof is analogous to Lemma \ref{lem:cylwp}.
\end{proof}

Note that in contrast to cylindrical normed spaces, a hyper-cylindrical normed space admits bar-joint frameworks which are not full.

We can now show that $K_7$ is the smallest graph which admits a full, well-positioned rigid and infinitesimally rigid bar-joint framework in hyper-cylindrical normed spaces.

\begin{theorem}
\label{thm:hcyl}
Let $(X,\|\cdot\|)$ be a hyper-cylindrical normed space.
\begin{enumerate}[(i)]
\item If~$(G,p)$ is   a  full, well-positioned and infinitesimally rigid bar-joint framework in $(X,\|\cdot\|)$, then either $G=K_7$ or $|V|\geq 8$.

\item For every $G\in \{K_m\colon m\ge7\}$, there is a placement $p$  in~$X$ so that $(G,p)$ is full, well-positioned and infinitesimally rigid in~$(X,\|\cdot\|)$. 
\end{enumerate}
\end{theorem}

\begin{proof}
(i) By Lemma \ref{HCylLemma}, we may assume, without loss of generality, that $(X,\|\cdot\|)=(\H_2(\bC),\|\cdot\|_{c_1})$.
By the Maxwell condition for $(\H_2(\bC),\|\cdot\|_{c_1})$ given in Theorem \ref{thm:maxwell}$(ii)$, we have $|E|=4|V|-7$. 
If $|V|\in\{3,4,5,6\}$ then, by Lemma \ref{lem:comb}, 
$|E|<4|V|-7$. 
The complete graph $K_2$ does not admit a full bar-joint framework in a hyper-cylindrical space. Thus $|V|\geq 7$.
If $|V|=7$ then $|E|=4|V|-7=21=\binom 72$ edges and so $G=K_7$.

$(ii)$ It is again sufficient to construct a suitable placement of $K_m$ in  
$(\bR^4,\|\cdot\|_{\hcyl})$.
First consider the case $m=7$. Let $\delta\in(1,\tfrac65)$ and $\epsilon\in(\tfrac\delta3,1-\tfrac\delta2)$,
and consider the placement $p:V\to\bR^4$ with
\begin{align*} 
  v_1&\mapsto(0,-1,-1,0),&
  v_2&\mapsto(0,1,-1,0)\\
  v_3&\mapsto(0,1,1,2\epsilon),&
  v_4&\mapsto(0,-1,1,-\delta)\\
  v_5&\mapsto (0,-1,1,2+\epsilon),& 
  v_6&\mapsto (0,1,-1,-2+3\epsilon)\\
	v_7&\mapsto (0,1,1,\delta).
\end{align*}
A calculation using Lemma~\ref{lem:hcylwp} shows that $(G,p)$ is well-positioned, with monochrome subgraphs~$G_1=\kappa_p^{-1}(\{1\})$ and~$G_2=\kappa_p^{-1}(\{2\})$ as indicated in Figure~\ref{fig:K7-colouring}. The graph $G_1$ is is an example of a block-and-hole graph, with one quadrilateral block and one quadrilateral hole and it follows from \cite[Theorem 4.1]{whi88} that $G_1$ is generically minimally rigid in $(\bR^3,\|\cdot\|_2)$. Alternatively, note that $G_1$ can be constructed from $K_4$ by successively adjoining vertices of degree three. It is a standard result that $K_4$ is generically minimally rigid in $(\bR^3,\|\cdot\|_2)$ (see for example \cite[Theorem 3.1]{whi84}), and that the graph operation of adjoining degree three vertices preserves generic minimal rigidity in $(\bR^3,\|\cdot\|_2)$ (\cite[Corollary 2.2]{whi84}).    
Also note that the monochrome subgraph $G_2$ is a spanning tree. 
By perturbing the vertices of $G$, we may assume that the projected monochrome subframework $(G_1,p_1)$ is regular, and hence minimally infinitesimally rigid, and also that $(G,p)$ is full. 
By Theorem~\ref{ProjThm}, $(G,p)$ is minimally infinitesimally rigid.
The argument from the proof of Theorem~\ref{thm:Km1} can now be adapted to show that $K_m$ admits a full, well-positioned and  infinitesimally rigid placement in $(\bR^4,\|\cdot\|_{\hcyl})$ for any $m>7$.
\end{proof}

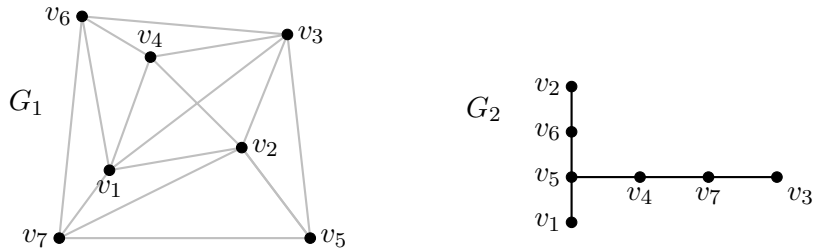
\begin{figure}[ht]
 \centering
 \begin{tabular}{  l  r  }

 \begin{minipage}{.37\textwidth}
 \begin{tikzpicture}[scale=1.2]
 
    \clip (-3,-2.25) rectangle (3.7cm,1.4cm); %
  
  \coordinate (A1) at (-1.2,-0.5);
  \coordinate (A2) at (0.25,-0.25);
  \coordinate (A3) at (0.75,1);
  \coordinate (A4) at (-0.75,0.75);
  \coordinate (A5) at (1,-1.25);
  \coordinate (A6) at (-1.5,1.2);
	\coordinate (A7) at (-1.75,-1.25);
		
  \draw[thick,lightgray] (A1) -- (A2) -- (A3) -- (A4) -- (A1) -- (A3) -- (A6) -- (A7) -- (A5) -- (A2) -- (A4) -- (A6) -- (A1);
  \draw[thick,lightgray] (A1) -- (A7) -- (A2) -- (A5) -- (A3);
	
  \node[draw,circle,inner sep=1.4pt,fill] at (A1) {};
  \node[draw,circle,inner sep=1.4pt,fill] at (A2) {};
  \node[draw,circle,inner sep=1.4pt,fill] at (A3) {};
  \node[draw,circle,inner sep=1.4pt,fill] at (A4) {};
	\node[draw,circle,inner sep=1.4pt,fill] at (A5) {};
  \node[draw,circle,inner sep=1.4pt,fill] at (A6) {};
	\node[draw,circle,inner sep=1.4pt,fill] at (A7) {};

  \node[left] at (-1.8,0.25) {$G_1$};
  
	\node[below] at (A1) {$v_1$};
  \node[right] at (A2) {$v_2$};
	\node[right] at (A3) {$v_3$};
  \node[above] at (A4) {$v_4$};
	\node[right] at (A5) {$v_5$};
  \node[left] at (A6) {$v_6$};	
	\node[left] at (A7) {$v_7$};
  \end{tikzpicture}
\end{minipage}

\hspace{18mm}
 \begin{minipage}{.37\textwidth}
 \begin{tikzpicture}[scale=1.2]
 
    \clip (-2.25,-0.7) rectangle (2.25cm,2cm); %
  
  \coordinate (A1) at (-1,0);
  \coordinate (A2) at (-1,1.5);
  \coordinate (A3) at (1.25,0.5);
  \coordinate (A4) at (-0.25,0.5);
  \coordinate (A5) at (-1,0.5);
  \coordinate (A6) at (-1,1);
	\coordinate (A7) at (0.5,0.5);

\draw[thick] (A2) -- (A6) -- (A5) -- (A4) -- (A7) -- (A3);
\draw[thick] (A5) -- (A1);

  \node[draw,circle,inner sep=1.4pt,fill] at (A1) {};
  \node[draw,circle,inner sep=1.4pt,fill] at (A2) {};
  \node[draw,circle,inner sep=1.4pt,fill] at (A3) {};
  \node[draw,circle,inner sep=1.4pt,fill] at (A4) {};
  \node[draw,circle,inner sep=1.4pt,fill] at (A5) {};
  \node[draw,circle,inner sep=1.4pt,fill] at (A6) {};
	\node[draw,circle,inner sep=1.4pt,fill] at (A7) {};

  \node[left] at (-1.65,1.25) {$G_2$};
 
	\node[left] at (A1) {$v_1$};
  \node[left] at (A2) {$v_2$};
	\node[below right] at (A3) {$v_3$};
  \node[below] at (A4) {$v_4$};
	\node[left] at (A5) {$v_5$};
  \node[left] at (A6) {$v_6$};
  \node[below] at (A7) {$v_7$};
		
  \end{tikzpicture}
\end{minipage}

\end{tabular}
\caption{The monochrome subgraphs $G_1$ and $G_2$ of $K_7$ constructed in Theorem~\ref{thm:hcyl} are respectively a $(3,6)$-tight, generically $3$-rigid, block-and-hole graph and a tree.}
\label{fig:K7-colouring}
\end{figure}

\bigskip

\end{document}